\documentclass[a4paper, 10pt]{article}

\usepackage{algorithmic}
\usepackage{algorithm}
\usepackage{amsmath}
\usepackage{amsfonts}
\usepackage{amssymb}
\usepackage{amsthm}
\usepackage{siunitx} 
\usepackage{booktabs}

\usepackage[utf8]{inputenc}
\usepackage{siunitx} 
\usepackage{setspace}
\usepackage[displaymath, mathlines]{lineno}

\newtheorem{theorem}{Theorem}

\newtheorem{lemma}{Lemma}

\newtheorem{corollary}{Corollary}

\newtheorem{condition}{Condition}
\usepackage{comment}
\usepackage[hidelinks]{hyperref} 
\usepackage[a4paper,margin=2.54cm]{geometry}
\title{On a reduction for a class of resource allocation problems}
\author{Martijn H. H. Schoot Uiterkamp, Marco E. T. Gerards, Johann L. Hurink \\ University of Twente, Enschede, the Netherlands}

\usepackage{multicol}

\begin{document}
\maketitle

\begin{abstract}
In the resource allocation problem (RAP), the goal is to divide a given amount of resource over a set of activities while minimizing the cost of this allocation and possibly satisfying constraints on allocations to subsets of the activities. Most solution approaches for the RAP and its extensions allow each activity to have its own cost function. However, in many applications, often the structure of the objective function is the same for each activity and the difference between the cost functions lies in different parameter choices such as, e.g., the multiplicative factors. In this article, we introduce a new class of objective functions that captures the majority of the objectives occurring in studied applications. These objectives are characterized by a shared structure of the cost function depending on two input parameters. We show that, given the two input parameters, there exists a solution to the RAP that is optimal for any choice of the shared structure. As a consequence, this problem reduces to the quadratic RAP, making available the vast amount of solution approaches and algorithms for the latter problem. We show the impact of our reduction result on several applications and, in particular, we improve the best known worst-case complexity bound of two important problems in vessel routing and processor scheduling from $O(n^2)$ to $O(n \log n)$.
\end{abstract}

\section{Introduction}
\label{sec_intro}

The resource allocation problem (RAP) is a classical problem within operations research and has been studied extensively and continuously since the 1950s \cite{Patriksson2008}. In its most basic and most studied form, this problem asks for the allocation of a given amount of resource over a set of activities while minimizing a given separable cost function (or, equivalently, maximizing a given separable utility function). Over the years, several variations and extensions of this basic setting have been studied, with different types of individual cost functions, additional constraints, and allocation restrictions such as integer-valued allocations \cite{Katoh2013}. 

With regard to the constraint structure, we focus on a general version of the RAP that occurs widely in applications, namely the RAP with additional submodular constraints (see, e.g., \cite{Groenevelt1991, Fujishige2005}). In this problem, for each subset of the activities, there is an upper bound on the total amount of resource allocated to these activities and this bound is given by a submodular set function. This problem has many applications in, e.g., machine learning \cite{Bach2010, Bach2013}, scheduling \cite{Shioura2018,Liu2020}, and game theory \cite{Jain2010, He2012,Harks2014}. Moreover, important special cases of this problem are the RAP with box constraints (see \cite{Patriksson2008}), the RAP with generalized bound constraints (see \cite{SchootUiterkamp2019a}), and the RAP with nested constraints (see \cite{Vidal2018}). Important application areas for in particular these special cases include, among many others, regularized learning \cite{Dai2006,Mairal2011}, telecommunications and energy management \cite{Palomar2005,vdKlauw2017, Xing2020}, and statistics \cite{Neyman1934,Friedrich2015} (see also the overviews in \cite{Patriksson2008} and \cite{Vidal2018}).

Concerning the objective, state-of-the-art solution approaches for RAPs generally allow each activity to have its own arbitrary (convex) cost function. Although this is an interesting aspect from a mathematical point of view, it is questionable whether this is a given situation in practical problems. In applications, often the structure of the cost functions is the same for all activities (e.g., all functions are quadratic) and the difference lies primarily in different parameters for these functions (e.g., each function has different multiplicative factors). In fact, this is the case for the large majority of applications studied in (the works surveyed in) \cite{Patriksson2008,Patriksson2015,Akhil2018,Vidal2018}. 

Scanning existing solution approaches for RAPs, one observes that, to obtain algorithms with a low computational complexity, many approaches rely on advanced data structures and procedures to store and manipulate problem parameters and intermediate bookkeeping values. However, it is often unclear whether such structures and procedures are actually fast in practice due to the lack of computational experiments (see also \cite{Patriksson2015}). Moreover, the fact that these algorithms are highly complex often makes it difficult to implement them, which limits their adaptability in practice where other aspects such as code maintainability and ease of use are often considered as more important (see also \cite{Muller-Hannemann2010}).

The aspects discussed above motivate us to study RAPs where the cost functions share a common structure and differ only in the parameter choice. More precisely, we introduce the notion of $(a,b,f)$-separable RAPs, wherein the cost function of each activity $i$ is of the form $a_i f (\frac{x_i}{a_i} + b_i)$, where $f$ is the common convex function and $a_i > 0$ and $b_i$ are given parameters that can be different for different activities~$i$. Such cost functions are an extension of so-called $d$-separable functions introduced in \cite{Veinott1971}. Moreover, they are closely related to the concept of perspective functions \cite{Combettes2018a, Combettes2018b}, which arise naturally in many problems in applied mathematics. In particular, most of the applications in (works surveyed in) \cite{Patriksson2008,Patriksson2015,Akhil2018,Vidal2018} can be modeled as $(a,b,f)$-separable RAPs.

In this article, we show a reduction result concerning $(a,b,f)$-separable RAPs with submodular constraints. More precisely, we show that for given parameters $a$ and $b$ and an instance of this class of RAPs, there exists a feasible solution to this instance that is optimal for any choice of the convex function $f$. In particular, we show that any solution that is optimal to the basic quadratic version of this RAP, i.e., where $f(x_i) = \frac{1}{2} x_i^2$, is also optimal for the $(a,b,f)$-separable version for any choice of~$f$. This means that solving any $(a,b,f)$-separable RAP reduces to solving the quadratic version of this RAP and allows us to solve this problem using any tailored algorithm that solves the quadratic RAP. Thus, to solve this problem, we do not require algorithms designed to solve the more general version with arbitrary convex cost functions, which are in general much slower and less efficient than the tailored algorithms for the quadratic RAP. Moreover, especially for the quadratic RAP over box constraints, many different types of algorithms exist to solve this problem, each of which has different pros and cons given the application \cite{Patriksson2008,Patriksson2015}. Thus, our reduction result allows us to solve a wide range of RAPs using the extensive collection of solution approaches and algorithms for quadratic RAPs.

In the literature, similar results already exist for specific RAPs. For RAPs over submodular constraints, \cite{Fujishige1980} showed that the problem with quadratic cost functions is equivalent to the problem of computing a lexicographically optimal base with regard to a given weight vector. \cite{Nagano2012} extends this result to a range of different strictly convex cost functions for the case of continuous variables. Their result is used in \cite{Nagano2013} to solve optimization problems on graphs and in \cite{Shioura2017} to derive efficient algorithms for processor scheduling problems. For a special case of RAPs with nested constraints, the equivalence of $(a,b,f)$-separable RAPs is proven in \cite{Akhil2018} for the case where the functions $f$ are strictly convex and differentiable, $b = 0$, and with continuous variables. 

Some reduction results can be derived from existing algorithms for specific applications in the literature. An example of this concerns the vessel speed optimization problem (see, e.g., \cite{Norstad2011}). In this problem, a ship traverses a given route between ports and must dock at each given port within a specific time window. The goal is to determine the ship's speed between each leg of the route while minimizing fuel costs. The authors in \cite{Norstad2011} propose a recursive-smoothing algorithm (RSA) for this problem, which is shown to be optimal by \cite{Hvattum2013}. This algorithm does not require knowledge on the fuel cost function other than that it is convex. Thus, the optimal solution outputted by this algorithm is indifferent of the choice of cost function.

Another example is the scheduling of tasks on a single processor with agreeable deadlines (see, e.g., \cite{Gerards2014}). Here, we are given a number of tasks that must be processed on a single processor, each of which has its own workload, arrival time, and deadline. The goal is to assign processor speeds to tasks such that all tasks are finished before their deadline while minimizing the total energy usage of the processor. This energy usage depends on the workload of each task and on the required power to maintain a given processor speed. Analogously to the vessel speed optimization problem, the processor scheduling problem can be solved using the RSA without any knowledge of the nature of the convex cost function \cite{Huang2009}. Thus, also the optimal solution outputted by this algorithm does not depend on the power function.

Our reduction result generalizes all the above results to general convex functions $f$, i.e., not necessarily \emph{strictly} convex or differentiable, and to both continuous and integer variables. In particular, in the case of continuous variables and a \emph{strictly} convex function~$f$, our reduction result becomes an equivalence result since the optimal solution to any strictly convex optimization problem is unique. In fact, given the parameters $a$ and $b$, an instance to RAP, and two strictly convex functions $f$ and $\bar{f}$, we show that the $(a,b,f)$-separable and $(a,b,\bar{f})$-separable versions of this RAP have the same unique optimal solution and are thus equivalent.

Next to the theoretical impact of our reduction result, we demonstrate the added value of this result for several applications. For a number of problems from the areas of telecommunications, statistics, and energy management, we show that our results provide new insights and improve several existing solution approaches. In particular, we show that the vessel speed optimization problem and the processor scheduling problem mentioned above can be solved in $O(n \log n)$ time. This is an improvement over their currently best known time complexity of $O(n^2)$.

Summarizing, our technical contributions are as follows:
\begin{enumerate}
\item We show that $(a,b,f)$-separable RAPs with submodular constraints reduce to their quadratic versions, making available the more extensive collection of fast and efficient algorithms for quadratic RAPs to solve these problems.
\item
We discuss the impact of this result on some special cases of the considered RAPs and derive new worst-case time complexity results for these cases.
\item
We apply our results to core problems from several application areas and show how they can be solved more efficiently using our reduction result. For two of these problems, we improve their worst-case time complexity from $O(n^2)$ to $O(n \log n)$.
\end{enumerate}
Moreover, on a higher level and perhaps of independent interest, our work demonstrates that methodological research on RAPs is conducted independently in many different research fields, be it under different names. As a consequence, many conceptual insights, structural properties, and solution approaches for RAPs have been re-invented and re-discovered many times over the years, both within the same field and independently in several fields. Therefore, we aim to promote a cross-disciplinary approach for studying RAPs. Such an approach will both reduce the amount of future re-discoveries and re-inventions and allow researchers to benefit from the many available different perspectives on RAPs.

The organization of this article is as follows. In Section~\ref{sec_pre}, we provide formal problem definitions of the studied RAPs and introduce the used notation. In Section~\ref{sec_equiv}, we prove the reduction result and in Section~\ref{sec_complexity}, we discuss the impact of this result on each of the studied RAPs. In Section~\ref{sec_appl}, we demonstrate the impact of our reduction result on several application areas. Finally, Section~\ref{sec_concl} contains our conclusions.

\section{Problem formulation and preliminaries}
\label{sec_pre}

In this section, we formulate the studied resource allocation problems, i.e., the RAP over submodular constraints and its special cases, and introduce the used notation and definitions. 

\subsection{Notation and definitions}
In the following, we introduce some notation and properties of  used functions and sets. For this, let $\mathcal{N} := \lbrace 1,\ldots,n \rbrace$ be the index set of the given activities. We call a convex function $\Phi : \ \mathbb{R}^n \rightarrow \mathbb{R}$ \emph{separable} if it can be written as the sum of single-variable convex functions, i.e., if $\Phi(x) = \sum_{i \in \mathcal{N}} \phi_i (x_i)$ for some single-variable convex functions $\phi_i : \ \mathbb{R} \rightarrow \mathbb{R}$, $i \in \mathcal{N}$. Moreover, given two vectors $a \in \mathbb{R}^n_{>0}$ and $b \in \mathbb{R}^n$ and a single-variable convex function~$f : \ \mathbb{R} \rightarrow \mathbb{R}$, we say that~$\Phi$ is \emph{$(a,b,f)$-separable}  if each function $\phi_i$ can be written as
\begin{equation*}
\phi_i(x_i) = a_i f \left(\frac{x_i}{a_i} + b_i \right).
\end{equation*}
Note that we do not pose any restrictions on~$f$ other than convexity. Hence, both $f$ and $\phi_i$ are not necessarily \emph{strictly} convex or differentiable. We denote the left and right derivatives of $f$ by $f^{-}$ and $f^{+}$ respectively. It follows that the left derivative $\phi_i^{-}$ of $\phi_i$ is given by
\begin{equation*}
\phi_i^{-}(x_i) := \lim_{y \uparrow x_i} \frac{\phi_i(y) - \phi_i(x_i)}{y - x_i}
=
\lim_{y \uparrow x_i} \frac{a_i f \left( \frac{y}{a_i} + b_i \right) - a_i f \left( \frac{x_i}{a_i} + b_i \right)}{y - x_i}
= f^{-}\left(\frac{x_i}{a_i} + b_i \right).
\end{equation*}
Analogously, the right derivative $\phi_i^{+}$ of $\phi_i$ is given by $\phi_i^{+} (x_i) =  f^{+}\left(\frac{x_i}{a_i} + b_i \right)$. Throughout this article, we call a resource allocation problem $(a,b,f)$-separable if its objective function is $(a,b,f)$-separable.

We denote the vector of ones of dimension~$n$ by $\bar{e}$. Furthermore, for each index $i \in \mathcal{N}$, we denote by $e^i \in \mathbb{R}^n$ the standard basis vector associated with~$i$, i.e., the vector whose $i^{\text{th}}$ entry is 1 and whose other entries are all zero. For a given set $\mathcal{C} \subset \mathbb{R}^n$ and $x \in \mathcal{C}$, let $\mathcal{E}_{\mathcal{C}}(x)$ denote the set of index pairs $(i,k) \in \mathcal{N}^2$ for which we can always shift a small amount from $x_i$ to $x_k$ without violating feasibility. More precisely, we define
\begin{equation*}
\mathcal{E}_{\mathcal{C}}(x) := \lbrace (i,k) \in \mathcal{N}^2 \ | \ x + \epsilon(e^k - e^i) \in \mathcal{C} \text{ for some } \epsilon > 0 \rbrace.
\end{equation*}
Each pair in $(i,k) \in \mathcal{E}_{\mathcal{C}}(x)$ is called an \emph{exchangeable pair}.

Finally, let $r: \ 2^{\mathcal{N}} \rightarrow \mathbb{R}$ be a set function on the ground set~$\mathcal{N}$. The set function~$r$ is \emph{submodular} if
$r(\mathcal{X} \cup \mathcal{Y}) + r(\mathcal{X} \cap \mathcal{Y}) \leq r(\mathcal{X}) + r(\mathcal{Y}) $ for any $\mathcal{X},\mathcal{Y} \subseteq \mathcal{N}$, where we assume that $r(\emptyset) = 0$.

\subsection{Problem classification}
The basic version of the resource allocation problem calls for an allocation $x \in \mathbb{R}^n$ of a given amount of resource $R \in \mathbb{R}$ over a set of activities $\mathcal{N}$ such that a given convex cost function $\Phi(x)$ of the allocation is minimized. This problem can be formulated as follows:
\begin{align*}
\text{RAP} \ : \ \min_x \ & \Phi(x) \\
\text{s.t. } & \sum_{i \in \mathcal{N}} x_i = R.
\end{align*}
Based on this basic version, we can formulate several extensions of the problem RAP with different types of cost functions, additional constraints, and different types of decision variables. To clearly distinguish between these problems, we adapt a classification scheme similar to that in \cite{Ibaraki1988} and \cite{Katoh2013}, i.e., we specify each problem by three fields $\alpha$/$\beta$/$\gamma$, where~$\alpha$ specifies the objective function, $\beta$ describes the constraint set, and $\gamma$ specifies the nature of the decision variables.

For $\alpha$, we consider the following options:
\begin{enumerate}
\item
\emph{Separable} (S): $\Phi$ is separable.
\item
\emph{$(a,b,f)$-separable} ($(a,b,f)$-S): $\Phi$ is $(a,b,f)$-separable.
\item
 \emph{Quadratic} ($(a,b)$-Q): $\Phi$ is $(a,b,f)$-separable using $f(y) = \frac{1}{2}y^2$. This means that $\Phi$ is both separable and quadratic.
\end{enumerate}

For $\beta$, we consider the follows special constraint structures, where we use $\mathcal{M} := \lbrace 1,\ldots, m \rbrace$ as an index set for additional constraints:
\begin{enumerate}
\item
\emph{Box constraints} (Box): $l_i \leq x_i \leq u_i$ for all $i \in \mathcal{N}$.
\item
\emph{Generalized bound constraints} (GBC): Next to the box constraints also constraints of the form $L_j \leq \sum_{i \in \mathcal{N}_j} x_i \leq U_j$, $j \in \mathcal{M}$ are given, where the sets $\mathcal{N}_1,\ldots,\mathcal{N}_m$ form a partition of $\mathcal{N}$.
\item
\emph{Nested constraints} (NC): Next to the box constraints also constraints of the form $L_j \leq \sum_{i \in \mathcal{N}_j} x_i \leq U_j$, $j \in \mathcal{M}$ are given, where the sets $\mathcal{N}_1,\ldots,\mathcal{N}_m$ are such that $\mathcal{N}_1 \subset \dots \subset \mathcal{N}_m \subset \mathcal{N}$.
\item
\emph{Laminar (or tree) constraints} (LC): Constraints of the form $L_j \leq \sum_{i \in \mathcal{N}_j} x_i \leq U_j$, $j \in \mathcal{M}$ are given, where the subsets $\mathcal{N}_1,\ldots,\mathcal{N}_m$ of $\mathcal{N}$ have the following property: if $\mathcal{N}_j \cap \mathcal{N}_{\ell} \neq \emptyset$, then either $\mathcal{N}_j \subset \mathcal{N}_{\ell}$ or $\mathcal{N}_j \supset \mathcal{N}_{\ell}$ for all $j,\ell \in \mathcal{M}$.
\item
\emph{Submodular constraints} (SC): Constraints of the form $\sum_{i \in \mathcal{S}} x_i \leq r(\mathcal{S})$, $\mathcal{S} \subset \mathcal{N}$ and $\sum_{i \in \mathcal{N}} x_i = r(\mathcal{N})$ are given, where $r$ is a given submodular function with $r(\mathcal{N}) = R$.
\end{enumerate}
Note that the constraint structures Box, GBC, and NC are special cases of the structure LC. Moreover, it can be shown that the structure LC is a special case of the structure SC (see Appendix~\ref{app_lemma_LC}). Thus, all constraint structures are special cases of SC. We discuss each of these special cases in more detail in Section~\ref{sec_complexity}.

For $\gamma$ we consider the following two cases:
\begin{enumerate}
\item
\emph{Continuous decision variables} (C): $x \in \mathbb{R}^n$.
\item
\emph{Integer decision variables} (I): $x \in \mathbb{Z}^n$.
\end{enumerate}

Table~\ref{tab_class} summarizes the possible entries of $\alpha$, $\beta$, and $\gamma$ as a compact reference. To simplify the presentation, we assume that whenever we specify $\beta$, the parameters that define the corresponding constraints are fixed. For example, when we consider the problems $(a,b,f)$-S/LC/C and $(a,b)$-Q/LC/C for some vectors $a,b$ and convex function $f$, we assume that the subsets $\mathcal{N}_1,\ldots, \mathcal{N}_m$ and vectors $L := (L_j)_{j \in \mathcal{M}}$ and $U := (U_j)_{j \in \mathcal{M}}$ are fixed.

\begin{table}[ht!]
\centering
\begin{tabular}{r l l}
\toprule
Field & Entry & Meaning \\
\midrule
$\alpha$ & S & Separable objective function \\
& $(a,b,f)$-S & $(a,b,f)$-separable objective function \\
& $(a,b)$-Q & $(a,b,f)$-separable objective function with $f(x_i) = \frac{1}{2}x_i^2$ \\
\midrule
$\beta$ & Box & Bounds on individual variables \\
& GBC & Bounds on individual variables and disjoint sums of variables \\
& NC & Bounds on individual variables and nested sums of variables \\
& LC & Bounds on laminar sums of variables \\
& SC & Bounds on sums of variables given by a submodular function \\
\midrule
$\gamma$ & C & Continuous variables ($x \in \mathbb{R}^n$) \\
& I & Integer variables ($x \in \mathbb{Z}^n$) \\
\bottomrule
\end{tabular}
\caption{Overview of entries for the problem classification $\alpha$/$\beta$/$\gamma$.}
\label{tab_class}
\end{table}

Finally, apart from these extensions of RAP, we study a general constraint class where a constraint $x \in \mathcal{C}$ for some set $\mathcal{C} \subset \mathbb{R}^n$ is given. We pose no restrictions on this set other than that it is nonempty. We denote this constraint class by $\mathcal{C}$ and denote the corresponding separable, $(a,b,f)$-separable, and $(a,b)$-quadratic versions of this problem by S/$\mathcal{C}$, $(a,b,f)$-S/$\mathcal{C}$, and $(a,b)$-Q/$\mathcal{C}$ respectively.

\section{Reduction of $(a,b,f)$-separable RAPs to quadratic RAPs}
\label{sec_equiv}

The goal of this section is to show for all the RAPs introduced in the previous section that their $(a,b,f)$-separable versions reduce to their quadratic versions. More precisely, given a constraint structure $\beta$, variable type $\gamma$, convex function $f$, and vectors $a \in \mathbb{R}^n_{>0}$ and $b \in \mathbb{R}^n$, we show that any optimal solution to $(a,b)$-Q/$\beta$/$\gamma$ is also optimal for $(a,b,f)$-S/$\beta$/$\gamma$. This means that we can solve $(a,b,f)$-S/$\beta$/$\gamma$ by solving $(a,b)$-Q/$\beta$/$\gamma$. Note that for many of these quadratic RAPs, tailored algorithms exist that are faster and more efficient than algorithms for the case with arbitrary convex cost functions. Thus, this reduction result allows us to solve $(a,b,f)$-S/$\beta$/$\gamma$ problems using fast algorithms for their quadratic special case.

We start by considering the general constrained optimization problem S/$\mathcal{C}$ with a convex separable objective function:
\begin{align*}
\text{S/}\mathcal{C} \ : \ \min_x \ & \sum_{i \in \mathcal{N}} \phi_i (x_i) \\
\text{s.t. } & x \in \mathcal{C},
\end{align*}
where $\mathcal{C} \subset \mathbb{R}^n$. Recall that we do not assume any properties on the set $\mathcal{C}$ other than that it is nonempty and that all RAPs introduced in the previous section are special instances of this problem.

We show that if S/$\mathcal{C}$ satisfies a certain optimality condition, any optimal solution to $(a,b)$-Q/$\mathcal{C}$ is also optimal for $(a,b,f)$-S/$\mathcal{C}$. This optimality condition states that a feasible solution $x$ to S/$\mathcal{C}$ is optimal if and only if moving an arbitrary amount from one variable $x_i$ to another variable $x_k$ while maintaining feasibility never leads to a decrease in objective value. We state this condition as Condition~\ref{prop_opt_cond} and give the mentioned reduction result in Theorem~\ref{th_same}.
\begin{condition}
Given separable convex functions $\phi_i  : \ \mathbb{R} \rightarrow \mathbb{R}$ and a set $\mathcal{C} \subset \mathbb{R}$, a feasible solution $x$ to S/$\mathcal{C}$ is optimal if and only if we have for each exchangeable pair $(i,k) \in \mathcal{E}_{\mathcal{C}}(x)$ that $\phi_k^{+} (x_k) \geq \phi_i^{-} (x_i)$.
\label{prop_opt_cond}
\end{condition}
\begin{theorem}
Let the set $\mathcal{C}$, a convex function~$f$, and $a \in \mathbb{R}^n_{>0}$ and $b \in \mathbb{R}^n$ be given. If Condition~\ref{prop_opt_cond} is satisfied by S/$\mathcal{C}$ and $x \in \mathcal{C}$ is optimal for $(a,b)$-Q/$\mathcal{C}$, then $x$ is also optimal for $(a,b,f)$-S/$\mathcal{C}$.
\label{th_same}
\end{theorem}
\begin{proof}
Let $x$ be an optimal solution to $(a,b)$-Q/$\mathcal{C}$. Note that for the problem $(a,b)$-Q/$\mathcal{C}$, we have that $\phi_i(x_i) = a_i \cdot \frac{1}{2}(\frac{x_i}{a_i} + b_i)^2 = \frac{1}{2} \frac{x_i^2}{a_i} + b_i x_i$ for all $i \in \mathcal{N}$. Thus, by applying Condition~\ref{prop_opt_cond} to $(a,b)$-Q/$\mathcal{C}$, we have that $\frac{x_k}{a_k} + b_k \geq \frac{x_i}{a_i} + b_i$ for all $(i,k) \in \mathcal{E}_{\mathcal{C}}(x)$. Since by convexity of $f$ the right derivative~$f^{+}$ is non-decreasing and we have $f^{+}(y) \geq f^{-}(y)$ for all $y \in \mathbb{R}$, it follows that $f^+\left(\frac{x_k}{a_k} + b_k \right) \geq f^+\left(\frac{x_i}{a_i} + b_i \right) \geq f^-\left(\frac{x_i}{a_i} + b_i \right)$ for all $(i,k) \in \mathcal{E}_{\mathcal{C}}(x)$. Note that this is equivalent to the statement $\phi_k^{+} (x_k) \geq \phi_i^{-} x_i$ for all $(i,k) \in \mathcal{E}_{\mathcal{C}}(x)$ where $\phi_{i'} (y) = a_{i'} f(\frac{x_{i'}}{a_{i'}} + b_{i'})$ for all $i' \in \mathcal{N}$. Thus, by applying Condition~\ref{prop_opt_cond} to $(a,b,f)$-S/$\mathcal{C}$, this implies that $x$ is optimal for $(a,b,f)$-S/$\mathcal{C}$.
\end{proof}
Note that Theorem~\ref{th_same} does not require the problem S/$\mathcal{C}$ to be a RAP. This means that this theorem and thus our reduction result is more widely applicable to other problems, provided that they satisfy Condition~\ref{prop_opt_cond}.

It is well-known that S/SC/$\gamma$ satisfies Condition~\ref{prop_opt_cond} for $\gamma \in \lbrace \text{C}, \text{I} \rbrace$ (see, e.g., \cite{Groenevelt1991, Fujishige2005}). To gain some insight in why this is the case, we provide for the interested reader in Appendix~\ref{sec_LC} an alternative proof for this claim for the relevant special case S/LC/$\gamma$ that relies only on basic concept from convex analysis such as subgradients. It follows that Theorem~\ref{th_same} can be applied to S/SC/$\gamma$ and in particular also to all special cases of this problem:
\begin{corollary}
Let a convex function~$f$, vectors $a \in \mathbb{R}^n_{>0}$, $b \in \mathbb{R}^n$, and entries $\beta$ and $\gamma$ as specified in Table~\ref{tab_class} be given. If $x$ is optimal for $(a,b)$-Q/$\beta$/$\gamma$, then $x$ is also optimal for $(a,b,f)$-S/$\beta$/$\gamma$.
\label{col_same}
\end{corollary}
This corollary is an extension of the equivalence results in \cite{Nagano2012}, where the reduction result is shown for the two special cases with continuous variables where $f$ is \emph{strictly} convex and differentiable or where $\phi_i = f$ for all $i \in \mathcal{N}$.

The validity of the reduction result of Theorem~\ref{th_same} for RAPs with submodular constraints and its special cases implies that any algorithm for solving the quadratic version of this problem can be used to solve the $(a,b,f)$-separable version. In particular, any time complexity or efficiency results for the quadratic version apply also to the $(a,b,f)$-separable version:
\begin{corollary}
Let a convex function~$f$, vectors $a \in \mathbb{R}^n_{>0}$, $b \in \mathbb{R}^n$, and entries $\beta$ and $\gamma$ as specified in Table~\ref{tab_class} be given. The worst-case time complexity of $(a,b,f)$-S/$\beta$/$\gamma$ equals that of $(a,b)$-Q/$\beta$/$\gamma$.
\label{col_complex}
\end{corollary}

Finally, for the case of continuous variables, Theorem~\ref{th_same} holds also when we replace the problem $(a,b)$-Q/$\beta$/C by $(a,b,\bar{f})$-S/$\beta$/C, where $\bar{f}$ is a \emph{strictly} convex function. This effectively turns our reduction result into an equivalence result between these two problems:
\begin{corollary}
Let $a \in \mathbb{R}^n_{>0}$, $b \in \mathbb{R}^n$, and entries $\beta$ as specified in Table~\ref{tab_class} be given, and let $\bar{f}$ be a \emph{strictly} convex function and $f$ be an arbitrary convex function. If $x$ is optimal for $(a,b,\bar{f})$-S/$\beta$/C, then $x$ is also optimal for $(a,b,f)$-S/$\beta$/C.
\label{col_strictly}
\end{corollary}
\begin{proof}
Since $\bar{f}$ is \emph{strictly} convex, $x$ is the \emph{unique} optimal solution to $(a,b,\bar{f})$-S/$\beta$/C. It follows from Theorem~\ref{th_same} that the unique optimal solution to $(a,b)$-Q/$\beta$/C is $x$ and thus that $x$ is also optimal for $(a,b,f)$-S/$\beta$/C.
\end{proof}
Corollary~\ref{col_strictly} allows us to solve a given continuous $(a,b,f)$-separable RAP using any algorithm that solves the $(a,b,\bar{f})$-separable version of the problem for some strictly convex function $\bar{f}$, i.e., not only just for quadratic objectives. This can be beneficial in cases where efficient algorithms have already been developed for a specific choice of a non-quadratic objective function, motivated by the given application.

In Section~\ref{sec_complexity}, we focus in more detail on each of the special cases of $\alpha$/SC/$\gamma$. In particular, using the reduction result in Theorem~\ref{th_same} and Corollary~\ref{col_complex}, we establish worst-case complexity results for the $(a,b,f)$-separable versions of these problems.

\section{Algorithms and complexity results for special cases}
\label{sec_complexity}

In this section, we first provide for each of the constraint types specified in Table~\ref{tab_class} a brief overview of known algorithms for the given special case and other known complexity results. In particular, we focus on algorithms and complexity results for the quadratic versions of these problems. Second, we use the complexity results on the quadratic versions of the problems to prove complexity results on the $(a,b,f)$-separable versions. These results are based on Theorem~\ref{th_same} and Corollary~\ref{col_same}, which state that we can solve each of these problems by solving the same problem with a quadratic objective function, i.e., where $f(y) = \frac{1}{2}y^2$.

As a compact reference, Tables~\ref{tab_complexity_Q} and~\ref{tab_complexity_S} summarize the complexity results discussed and obtained in this section.
\begin{table}[ht!]
\centering
\begin{tabular}{r  l l }
\toprule
$\beta$ & \multicolumn{2}{c}{$\gamma$}\\\cmidrule{2-3}
& C & I  \\
\midrule
Box & $O(n)$ & $O(n)$  \\
GBC & $O(n)$ & $O(n)$  \\
NC & $O(n \log m)$ & $O(n \log m)$  \\
LC & $O(n^2)$, $O(n \log n)$ (only upper constraints)  & $O(n^2)$  \\
SC & $O(n^2 + n \cdot \text{EO})$ (decomposition), &
$O(n^2 F \log r(\mathcal{N}) + n\tilde{F})$ (decomposition),
\\
&
$ O(n (\log n + \tilde{F}) \log \frac{r(\mathcal{N})}{\epsilon n})$ (greedy) &$O(n (\log n + \tilde{F}) \log \frac{r(\mathcal{N})}{n})$ (greedy) \\
\bottomrule
\end{tabular}
\caption{Overview of the worst-case time complexity results for the problems $(a,b)$-Q/$\beta$/$\gamma$ and $(a,b,f)$-S/$\beta$/$\gamma$.}
\label{tab_complexity_Q}
\end{table}
\begin{table}[ht!]
\centering
\begin{tabular}{r  l l }
\toprule
$\beta$ & \multicolumn{2}{c}{$\gamma$}\\\cmidrule{2-3}
 & C & I  \\
\midrule
Box &  $O(n \log \frac{nR}{\epsilon})$ & $O(n \log \frac{R}{n})$ \\
GBC  & $O(n \log \frac{nR}{\epsilon})$ & $O(n \log \frac{R}{n})$ \\
NC &  $O(n \log m \log \frac{nR}{\epsilon})$ & $O(n \log m \log R)$ \\
LC    & $O(n^2 \log n \log \frac{nR}{\epsilon})$,& $O(n^2 \log n \frac{mR}{n})$, \\
& $O(n \log n \log \frac{nR}{\epsilon})$ (only upper constraints) & $O(n \log n \log \frac{R}{n})$ (only upper constraints) \\
SC & $O(n^2 \log \frac{n r(\mathcal{N})}{\epsilon} + n \cdot EO)$ (decomposition), & $O(n^2 (\log \frac{r(\mathcal{N})}{n} + F \log r(\mathcal{N}) )+ n \tilde{F})$ (decomposition), \\
& 
 $O(n (\log n + \tilde{F}) \log \frac{r(\mathcal{N})}{\epsilon n})$ (greedy) & $O(n (\log n + \tilde{F}) \log \frac{r(\mathcal{N})}{n})$ (greedy)\\
\bottomrule
\end{tabular}
\caption{Overview of the worst-case time complexity results for the problem S/$\beta$/$\gamma$.}
\label{tab_complexity_S}
\end{table}

\subsection{$\alpha$/Box/$\gamma$: Optimization over a single linear constraint}
\label{sec_box}
The resource allocation problem over a single linear constraint, $(a,b)$-S/Box/$\gamma$, can be formulated as follows:
\begin{align}
(a,b)\text{-S/Box/}\gamma \ : \ \min_{x} \ & \sum_{i \in \mathcal{N}} a_i f \left(\frac{x_i}{a_i} + b_i \right) \nonumber \\
\text{s.t. } & \sum_{i \in \mathcal{N}} x_i = R, \label{eq_box_02}\\
& l_i \leq x_i \leq u_i, \quad i \in \mathcal{N}, \label{eq_box_03} \\
& x \in      \begin{cases} 
        \mathbb{R}^n &  \text{if } \gamma = \text{C},\\
        \mathbb{Z}^n   & \text{if } \gamma = \text{I}. \\
    \end{cases} \nonumber
\end{align}
This problem and its more general version S/Box/$\gamma$ have been studied since the 1950s \cite{Patriksson2008}. Since then, many solution approaches and algorithms have been proposed for this problem, especially for the problems Q/Box/$\gamma$. We refer to \cite{Patriksson2008, Patriksson2015} for surveys on the continuous version S/Box/C and to \cite{Katoh2013} for a brief but thorough review on the integer version S/Box/I.

The best known complexities for S/Box/C and S/Box/I are $O(n \log \frac{nR}{\epsilon})$ and $O(n \log \frac{R}{n})$ respectively, where $\epsilon$ is an accuracy parameter  \cite{Frederickson1982, Hochbaum1994}. Furthermore, their quadratic versions $(a,b)$-Q/Box/C and $(a,b)$-Q/Box/I can be solved in $O(n)$ time (\cite{Brucker1984} and \cite{Ibaraki1988} respectively). Through Corollary \ref{col_complex}, this yields the following complexity results for $(a,b,f)$-S/Box/$\gamma$:
\begin{corollary}
Both $(a,b,f)$-S/Box/C and $(a,b,f)$-S/Box/I can be solved in $O(n)$ time.
\label{col_box_complexity}
\end{corollary}
The linear-time algorithms for $(a,b)$-Q/Box/C belong to the class of so-called \emph{breakpoint search} algorithms that solve the problem by efficiently searching for the optimal Lagrange multiplier corresponding to the resource constraint~(\ref{eq_box_02}) (see also \cite{Kiwiel2008a}). The linear-time algorithm for $(a,b)$-Q/Box/I in \cite{Ibaraki1988} first solves the continuous version $(a,b)$-Q/Box/C of this problem using a linear-time algorithm such as in \cite{Brucker1984}. Subsequently, it uses this solution and a specific rounding scheme to construct an instance of $(a,b)$-Q/Box/I with $R = O(n)$ that has the same optimal solution as the original instance of $(a,b)$-Q/Box/I. Using the algorithm in, e.g., \cite{Frederickson1982,Hochbaum1994} for S/Box/I, this instance can be solved in $O(n \log \frac{O(n)}{n}) = O(n)$ time.

With regard to practical execution time, there are several classes of algorithms that outperform the aforementioned linear-time algorithms. For example, for the problem $(a,b)$-Q/Box/C, \cite{Kiwiel2008b} shows that so-called variable-fixing algorithms that run in $O(n^2)$ time are in general faster than linear-time algorithms such as in \cite{Brucker1984}. These algorithms first compute a solution to the problem without the box constraints~(\ref{eq_box_03}) and subsequently determine the optimal value of several variables that exceed their bounds in this solution. This process continues until none of the variables in the solution to the relaxed problem exceeds its bounds. The worst-case time complexity of $O(n^2)$ is attained when only one variable can be fixed to its optimal value during each step in the procedure. However, this is quite a pathological case since it has the property that in the optimal solution all variables are equal to one of their bounds.

Moreover, \cite{Wright2014} shows that for several instances of $(a,b,f)$-S/Box/C, a specialized interior-point method significantly outperforms other approaches including the linear-time breakpoint search approaches. Interior-point methods are iterative approaches where each intermediate solution is obtained from the previous one by taking a step in a search direction that is the solution of a perturbed version of the Karush-Kuhn-Tucker optimality conditions (see also \cite{Gondzio2012}). Normally, the computation of this search direction is the computationally most expensive step of the interior-point method since it requires solving a linear system involving the constraint matrix. However, by exploiting the sparse structure of the constraint matrix for S/Box/C, the number of operations required to solve this system can be reduced from $O(n^3)$ to $O(n)$.

One reason for the in practice quite bad practical performance of linear-time algorithms for $(a,b)$-Q/Box/C is that they require the computation of the median of sets of numbers. However, to attain a linear-time complexity, also linear-time procedures for median finding such as in \cite{Blum1973} have to be used. Such methods are in general significantly slower than alternative sorting-based approaches that run in linearithmic time \cite{Kiwiel2005, Alexandrescu2017}.

The linear-time complexity of $(a,b)$-Q/Box/I is based on the linear-time complexity of $(a,b)$-Q/Box/C and the existence of linear-time algorithms for selecting a $k^{\text{th}}$ smallest element from a collection of sorted lists \cite{Ibaraki1988}. For the latter problem, many studies refer to \cite{Frederickson1982} for such a linear-time algorithm. Analogously to the breakpoint search algorithms for $(a,b)$-Q/Box/C, this algorithm requires a linear-time algorithm for median-finding to attain a linear-time complexity and may thus be slower in practice than alternative sorting-based approaches. It should be noted, however, that recently new linear-time algorithms have been developed that are based on specialized heap data structures and have been shown to have a better practical performance (see, e.g., \cite{Kaplan2019}).

\subsection{$\alpha$/GBC/$\gamma$: Optimization over generalized bound constraints}
Let $\mathcal{N}_1,\ldots,\mathcal{N}_m$ be a partition of the index set $\mathcal{N}$. Given parameters $L,U \in \mathbb{R}^m$, the resource allocation problem with generalized bound constraints can be formulated as
\begin{align}
(a,b)\text{-S/GBC/$\gamma$} \ : \ \min_{x} \ & \sum_{i \in \mathcal{N}} a_i f \left(\frac{x_i}{a_i} + b_i \right) \nonumber \\
\text{s.t } & \sum_{i \in \mathcal{N}} x_i = R, \nonumber \\
& L_j \leq \sum_{i \in \mathcal{N}_j} x_i \leq U_j, \quad j \in \mathcal{M}, \label{GBC_03} \\
& l_i \leq x_i \leq u_i, \quad i \in \mathcal{N}, \nonumber \\
& x \in      \begin{cases} 
        \mathbb{R}^n &  \text{if } \gamma = \text{C},\\
        \mathbb{Z}^n   & \text{if } \gamma = \text{I}. \\
    \end{cases} \nonumber
\end{align} 
Applications of this problem include portfolio optimization \cite{Lobo2007}, transportation problems \cite{Cosares1994}, stratified sampling \cite{Sanathanan1971}, and electric vehicle charging \cite{SchootUiterkamp2019a}.

In the literature, this problem is studied primarily with only the upper bound constraints in~(\ref{GBC_03}). \cite{Hochbaum1994} shows that S/GBC/$\gamma$ with only generalized \emph{upper} bound constraints can be solved in the same time as S/Box/$\gamma$ by reducing the problem to a sequence of subproblems S/Box/$\gamma$ over in total $n$ variables. \cite{SchootUiterkamp2019a} shows a similar result for $(a,b)$-Q/GBC/$\gamma$ with both generalized lower and upper bound constraints, which yields an $O(n)$ algorithm for solving $(a,b)$-Q/GBC/$\gamma$. Thus, by Corollary~\ref{col_complex}, also the problems $(a,b,f)$-S/GBC/$\gamma$ can be solved in $O(n)$ time:

\begin{corollary}
Both $(a,b,f)$-S/GBC/C and $(a,b,f)$-S/GBC/I can be solved in $O(n)$ time.
\end{corollary}

Alternatively, the continuous problem $(a,b)$-Q/GBC/C can be solved in $O(n)$ time as a special case of quadratic programming with a fixed number of constraints \cite{Megiddo1993}.

\subsection{$\alpha$/NC/$\gamma$: Optimization over nested constraints}
\label{sec_NC}
Let $\mathcal{N}_1,\ldots, \mathcal{N}_m$ be subsets of $\mathcal{N}$ such that $\mathcal{N}_1 \subset \dots \subset \mathcal{N}_m \subset \mathcal{N}$. The resource allocation problem with nested constraints, $(a,b)$-S/NC/$\gamma$, is stated as follows:
\begin{align}
(a,b)\text{-S/NC/$\gamma$} \ : \ \min_{x} \ & \sum_{i \in \mathcal{N}} a_i f \left(\frac{x_i}{a_i} + b_i \right) \nonumber \\
\text{s.t } & \sum_{i \in \mathcal{N}} x_i = R, \nonumber \\
& L_j \leq \sum_{i \in \mathcal{N}_j} x_i \leq U_j, \quad j \in \mathcal{M}, \label{NC_03}\\
& l_i \leq x_i \leq u_i, \quad i \in \mathcal{N}, \label{NC_04} \\
& x \in      \begin{cases} 
        \mathbb{R}^n &  \text{if } \gamma = \text{C},\\
        \mathbb{Z}^n   & \text{if } \gamma = \text{I}. \\
    \end{cases} \nonumber
\end{align}

Research on this problem and the more general problem S/NC/$\gamma$ has almost exclusively focused on the case with either the lower or upper nested constraints in~(\ref{NC_03}) but not both. We refer to \cite{Akhil2018} for a survey on this version of the problem.

The most efficient algorithm for both S/NC/$\gamma$ and $(a,b)$-Q/NC/$\gamma$ is the decomposition algorithm in \cite{Vidal2018}. This algorithm solves the problem as a sequence of S/Box/$\gamma$ subproblems where the single-variable bounds~(\ref{eq_box_03}) of each subproblem are optimal solutions to subproblems deeper in the decomposition hierarchy. The worst-case time complexity of this algorithm is $O(n \log m \log \frac{nR}{\epsilon})$ for S/NC/C, $O(n \log m \log R)$ for S/NC/I, and $O(n \log m)$ for both $(a,b)$-Q/NC/C and $(a,b)$-Q/NC/I. Thus, it follows directly from Corollary~\ref{col_complex} that both $(a,b,f)$-S/NC/C and $(a,b,f)$-S/NC/I can be solved in $O(n \log m)$:
\begin{corollary}
Both $(a,b,f)$-S/NC/C and $(a,b,f)$-S/NC/I can be solved in $O(n \log m)$ time.
\label{col_NC}
\end{corollary}

The algorithm in \cite{Vidal2018} attains the $O(n \log m)$ time complexity by utilizing the linear-time algorithms for $(a,b)$-Q/Box/$\gamma$ to solve the subproblems. As mentioned in Section~\ref{sec_box}, these are not the fastest algorithms for these subproblems. As a consequence, it can be expected that using, e.g., variable-fixing algorithms \cite{Kiwiel2008b} for the subproblems significantly improves the overall execution time of the algorithm.

It has been shown \cite{Wu2019,SchootUiterkamp2020b} that infeasibility-guided algorithms such as in \cite{vdKlauw2017,Wu2019} are significantly faster than the decomposition algorithm in \cite{Vidal2018}. These algorithms first compute a solution to S/NC/$\gamma$ without the nested constraints~(\ref{NC_03}) and, based on which nested constraint is violated most in this solution, subsequently divide the problem into two smaller instances of this problem. Analogously to the variable-fixing algorithms for $(a,b)$-Q/Box/C, the maximum number of divisions is $O(n)$, which results in a worst-case time complexity of $O(n^2 \log \frac{nR}{\epsilon})$ for S/NC/C \cite{vdKlauw2017} and $\Theta(n^2 \log \frac{R}{n})$ for S/NC/I \cite{Wu2019}. However, this worst-case complexity occurs only in pathological cases where each nested constraint is tight in an optimal solution, whereas it can be expected that the number of tight constraints is relatively small in practice. In particular, for the case with only upper nested constraints~(\ref{NC_03}), lower single-variable bounds~(\ref{NC_04}), and randomly generated problem parameters, it is shown in \cite{Vidal2016} that the expected number of tight constraints in an optimal solution to $(a,b,f)$-S/NC/C is $O(\log n)$.

An alternative algorithm for $(a,b)$-Q/NC/C that attains the same time complexity as \cite{Vidal2018} for $m=n$ is given in \cite{SchootUiterkamp2020b}. This algorithm is similar to the decomposition algorithm of \cite{Vidal2018} in the sense that it solves a (slightly different) sequence of $(a,b)$-Q/Box/C subproblems where the single-variable bounds for each subproblem are optimal solutions to previous subproblems. However, this algorithm avoids the time-consuming explicit computation of solutions to subproblems by exploiting the properties of a specific breakpoint searching algorithm for $(a,b)$-Q/Box/C and computing only the optimal Lagrange multiplier of each subproblem. As a consequence, this algorithm is shown to be one order of magnitude faster than the decomposition algorithm of \cite{Vidal2018}, while attaining the same worst-case time complexity of $O(n \log n)$ for $m = O(n)$.

Recently, for the problem $(a,b)$-Q/NC/C with only upper nested constraints, \cite{Wright2020} shows that a specialized interior-point method is able to outperform the decomposition-based approach in \cite{Vidal2016}, which is similar to the approach in \cite{Vidal2018}, when the ratio $\frac{m}{n}$ is larger than 0.1. Analogously to \cite{Wright2014} as mentioned in Section~\ref{sec_box}, this method exploits the constraint structure of S/NC/C to compute search directions in $O(n)$ time instead of $O(n^3)$ time. Although the authors in \cite{Wright2020} consider only upper nested constraints, it is straight-forward to generalize their results to problems involving also lower nested constraints \cite{Slager2019}.

Interestingly, \cite{Vidal2016,Akhil2018} shows that we can solve the problem $(a,b,f)$-S/NC/C with only nested upper constraints and without the box constraints~(\ref{NC_04}) in $O(n)$ time. More precisely, they show that this problem can be reduced to the problem of finding a concave cover of~$n$ points in $\mathbb{R}^2$ and give an $O(n)$ time algorithm to find this cover. This algorithm is very similar to the recursive-smoothing algorithm mentioned in Section~\ref{sec_intro} that is used to solve the vessel speed optimization problem \cite{Norstad2011} and processor scheduling problem with agreeable deadlines \cite{Huang2009}.

\subsection{$\alpha$/LC/$\gamma$: Optimization over laminar constraints}
Let $\mathcal{N}_1,\ldots,\mathcal{N}_m$ be subsets of $\mathcal{N}$ that satisfy the following property: if $\mathcal{N}_j \cap \mathcal{N}_{\ell} \neq \emptyset$, then either $\mathcal{N}_j \subset \mathcal{N}_{\ell}$ or $\mathcal{N}_j \supset \mathcal{N}_{\ell}$ for all $j,\ell \in \mathcal{M}$. We formulate the resource allocation with laminar constraints, $(a,b,f)$-S/LC/$\gamma$, as follows:
\begin{align}
(a,b,f)\text{-S/LC/$\gamma$} \ : \ \min_{x} \ & \sum_{i \in \mathcal{N}} a_i f \left(\frac{x_i}{a_i} + b_i \right) \nonumber \\
\text{s.t } & \sum_{i \in \mathcal{N}} x_i = R, \nonumber \\
 & L_j \leq \sum_{i \in \mathcal{N}_j} x_i \leq U_j, \quad j \in \mathcal{M}, \label{LC_03} \\
& l_i \leq x_i \leq u_i, \quad i \in \mathcal{N}, \nonumber \\
& x \in      \begin{cases} 
        \mathbb{R}^n &  \text{if } \gamma = \text{C},\\
        \mathbb{Z}^n   & \text{if } \gamma = \text{I}. \\
    \end{cases} \nonumber
\end{align} 
Similarly to S/NC/$\gamma$, the problem S/LC/$\gamma$ has been studied mainly with only the upper laminar constraints in~(\ref{LC_03}). The algorithms with the lowest computational complexities for these problems are given by \cite{Hochbaum1994} and have time complexities of $O(n \log n \log \frac{nR}{\epsilon})$ for $\gamma = \text{C}$ and $O(n \log n \log \frac{R}{n})$ for $\gamma = \text{I}$. For the general problem S/LC/$\gamma$, we obtain an efficient algorithm by combining results on the complexity of general separable convex optimization problems with linear constraints \cite{Hochbaum1990} and of the problem S/LC/C with a linear objective function \cite{Orlin2013}. More precisely, the time complexities of S/LC/C and S/LC/I are $O(P_{\text{linear}}(8n^2,m)\log \frac{Rn}{\epsilon})$ and $O(P_{\text{linear}}( 4n^2,m) \log \frac{Rm}{n})$ respectively, where $P_{\text{linear}}(n,m)$ is the time complexity of solving an instance of S/LC/C with a linear objective function \cite{Hochbaum1990}. The latter problem can be solved in $O(n \log n)$ time using the algorithm in \cite{Orlin2013}, hence we obtain a time complexity of $O(n^2 \log n \log \frac{Rn}{\epsilon})$ and $O(n^2 \log n \log \frac{Rm}{n})$ for S/LC/C and S/LC/I respectively.

With regard to the quadratic version of the problem, the special case of $(a,b)$-Q/LC/C with only upper laminar constraints can be solved in $O(n \log n)$ time \cite{Hochbaum1995}. This is done by reducing the problem to an instance of $(a,b)$-Q/NC/C with only upper nested constraints, which can be solved in $O(n \log n)$ time \cite{Hochbaum1995}. The general version of $(a,b)$-Q/NC/C with both lower and upper laminar constraints can be solved in $O(n^2)$ time as an instance of the quadratic convex cost flow problem on a tree network \cite{Tamir1993}. Finally, the integer-valued problem $(a,b)$-Q/NC/I can be solved in $O(n^2)$ time by first computing a solution to the continuous version of this problem and subsequently using a specific rounding procedure to obtain the optimal integer solution from this continuous solution \cite{Moriguchi2011}. By Corollary~\ref{col_complex}, this yields the following worst-case time complexities for $(a,b,f)$-S/LC/C and $(a,b,f)$-S/LC/I:

\begin{corollary}
Problem $(a,b,f)$-S/LC/$\gamma$ can be solved in $O(n^2)$ time. The special case $(a,b)$-Q/LC/C with only upper laminar constraints can be solved in $O(n \log n)$ time.
\end{corollary}

As far as we are aware, the problem S/LC/$\gamma$ has been studied primarily from an academic point of view in the literature, i.e., little attention is paid to possible applications. One relevant application that has received quite some importance in the past years is the scheduling of the (dis)charging of an electrical storage system within a smart grid (see also Section~\ref{sec_storage}) where the energy can be drawn from each of the three phases within the low-voltage distribution network (see also \cite{SchootUiterkamp2019a}). The resulting problem is an instance of S/LC/C where the feasible set is the intersection of nested constraints (to model the storage capacity limits) and generalized upper bound constraints (to model the charging limits). We plan to investigate this topic further in future research.

\subsection{$\alpha$/SC/$\gamma$: Optimization over submodular constraints}
Given a submodular function $r$ over the ground set $\mathcal{N}$, the $(a,b,f)$-separable resource allocation over submodular constraints can be formulated as follows:
\begin{align}
(a,b,f)\text{-S/SC/}\gamma \ : \ \min_x \ & \sum_{i \in \mathcal{N}} a_i f \left(\frac{x_i}{a_i} + b_i \right) \nonumber \\
\text{s.t } & \sum_{i \in \mathcal{N}} x_i = r(\mathcal{N}), \nonumber \\
&  \sum_{i \in \mathcal{S}} x_i \leq r(\mathcal{S}), \quad \mathcal{S} \subset \mathcal{N}, \label{P_SC_03} \\
& x \in      \begin{cases} 
        \mathbb{R}^n &  \text{if } \gamma = \text{C},\\
        \mathbb{Z}^n   & \text{if } \gamma = \text{I}. \\
    \end{cases} \nonumber
\end{align}

For this problem, one can find two classes of algorithms in the literature. The first class consists of decomposition algorithms that first compute a solution to the problem without the submodular constraints~(\ref{P_SC_03}) and, based on which constraints are violated by this solution, split up the problem into two smaller instances of S/SC/$\gamma$ \cite{Fujishige1980,Groenevelt1991}. Note that the infeasibility-guided algorithms for S/NC/$\gamma$ as discussed in Section~\ref{sec_NC} are based on the same principle. The best worst-case time complexities of such algorithms are $O(n^2 \log \frac{n r(\mathcal{N})}{\epsilon} + n \cdot EO)$ for S/SC/C \cite{Nagano2012} and $O(n^2( \log \frac{r(\mathcal{N})}{n} +  F \log r(\mathcal{N})) + n \tilde{F})$ for S/SC/I \cite{Katoh2013}, where EO is the time required to minimize a given submodular function and $F$ is the time required to check the feasibility of a given vector for the submodular constraints. Moreover, $\tilde{F}$ is the time required to determine for a given solution $x$ that is feasible for the submodular constraints~(\ref{P_SC_03}) by how much we can increase a given variable $x_i$ without violating any of these submodular constraints. For the quadratic problems $(a,b)$-Q/SC/$\gamma$, these complexities reduce to $O(n^2 + n \cdot \text{EO})$ for $(a,b)$-Q/SC/C and to $O(n^2 F \log r(\mathcal{N}) + n\tilde{F})$ for $(a,b)$-Q/SC/I. By Corollary~\ref{col_same}, these are also the complexities for solving the problems $(a,b,f)$-S/SC/C and $(a,b,f)$-S/SC/I using decomposition algorithms.

The second class consists of greedy algorithms that solve the integer version S/SC/I by incrementally building an optimal solution (see, e.g., \cite{Hochbaum1994, Moriguchi2004}). However, instead of incrementing the total amount of allocated resource by unit steps, these algorithms apply a scaling procedure to determine larger step sizes that speed up the building process while still maintaining feasibility of the current solution. To solve the continuous version S/SC/C, these algorithms exploit a proximity result between optimal solutions of S/SC/C and S/SC/I (see, e.g., \cite{Moriguchi2011}) that states that for any optimal solution $x^*$ to S/SC/I there exists an optimal solution $\tilde{x}$ to S/SC/C such that $|\tilde{x}_i - x^*_i| \leq n-1$. As a consequence, to solve S/SC/C with an given accuracy $\epsilon$, one can scale all problem parameters by a factor $ \lceil \frac{n}{\epsilon} \rceil$, solve the scaled problem with integer variables using the greedy algorithm, and scale back the resulting solution. The most efficient algorithms of this class run in $O(n (\log n + \tilde{F}) \log \frac{r(\mathcal{N})}{\epsilon n})$ time for S/SC/C and $O(n (\log n + \tilde{F}) \log \frac{r(\mathcal{N})}{n})$ for S/SC/I \cite{Hochbaum1994,Moriguchi2004}, which unfortunately cannot be improved for the quadratic cases $(a,b)$-Q/SC/C and $(a,b)$-Q/SC/I.

One relevant special case of $(a,b)$-Q/SC/C is the problem of computing the minimum-norm point of a base polytope (see, e.g., \cite{Fujishige2005}). This problem is equivalent to $(\bar{e},0)$-Q/SC/C and plays an important role as a subroutine in several algorithms for machine learning problems and submodular function minimization \cite{Fujishige2011, Bach2013}. One of the most popular algorithms in practice for finding the minimum-norm point is Wolfe's algorithm \cite{Wolfe1976}, which solves the problem by iteratively updating a hyperplane and the minimum-norm point on this hyperplane based on the feasibility of this point. The authors in \cite{Chakrabarty2014} show that this algorithm computes an $\epsilon$-approximate solution to $(\bar{e},0)$-Q/SC/C in $O(\frac{nM^2}{\epsilon})$ time, where $M$ is the norm of the \emph{maximum}-norm point. Although there are algorithms for finding the minimum-norm point that have a better computational complexity, e.g., the aforementioned decomposition and greedy algorithms, Wolfe's algorithm has been shown to be among the fastest algorithms in practice \cite{Fujishige2011, Bach2013}.

\section{Impact on applications}
\label{sec_appl}

The goal of this section is to show the relevance of $(a,b,f)$-separable resource allocation problems in applications. As we discussed in the previous section, our newly derived complexity results might not directly lead to practical faster algorithms for these problems. However, for a number of applications from the domains of telecommunications, statistics, and energy management, we show that our reduction result lead to new insights into common practices in these fields. In particular, we show that two problems in the area of vessel routing and processor scheduling can be solved in $O(n \log n)$ time rather than $O(n^2)$ time, which was the previously known best complexity for these problems. Finally, with this collection of applications and the included references, we intent to stimulate cross-disciplinary research that leads to new structural results and algorithms for RAPs that are applicable to many different research fields.

\subsection{Power allocation in multi-channel communication systems}

In many telecommunication systems, data can be transmitted over several parallel channels to reduce the amount of noise experienced when transmitting the data (see, e.g., \cite{Shams2014}). The amount of data that can be transmitted through a given channel~$i$, i.e., the channel capacity, depends on the power $x_i$ spent on this channel, its bandwidth~$B_i$, and a ``gain'' parameter~$c_i$ that represents the amount of noise on the channel. One goal in these systems is to allocate a given budget of total power $P^{\text{tot}}$ over a set $\mathcal{N}$ of~$n$ channels such that the overall channel capacity is maximized while respecting power limits on each channel. This problem can be formulated mathematically as
\begin{align*}
\text{(P)} \ : \ \max_{x \in \mathbb{R}^n} \ & \sum_{i \in \mathcal{N}} B_i \log (1 + c_i x_i) \\
\text{s.t. } & \sum_{i \in \mathcal{N}} x_i = P^{\text{tot}}, \\
& 0 \leq x_i \leq \bar{P}_i, \quad i \in \mathcal{N},
\end{align*}
where $\bar{P}_i$ is the maximum allowed power on channel~$i$.

Note that for a given channel~$i \in \mathcal{N}$ we have
\begin{equation*}
B_i \log (1 + c_i x_i)
=
B_i \log  \left( \frac{\frac{1}{c_i} + x_i}{B_i} \cdot B_i c_i \right)
=
B_i \log \left( \frac{\frac{1}{c_i} + x_i}{B_i} \right)
+ B_i \log(B_i c_i).
\end{equation*}
Since the second term $B_i \log(B_i c_i)$ in the above expression is constant, we can replace the objective function of Problem~(P) by $\sum_{i \in \mathcal{N}} B_i \log \left( \frac{\frac{1}{c_i} + x_i}{B_i} \right)$ without changing the optimal solution to the problem. The resulting problem is an instance of $(B,\bar{B},f)$-S/Box/C with $B := (B_i)_{i \in \mathcal{N}}$, $\bar{B} := (\frac{1}{B_i c_i})_{i \in \mathcal{N}}$, and $f(x_i) := -\log(x_i)$. Thus, by Corollaries~\ref{col_same} and~\ref{col_box_complexity}, we can solve this problem as an instance of $(B,\bar{B})$-Q/Box/C in $O(n)$ time, i.e., we can replace each term $B_i \log (1 + c_i x_i)$ by $\frac{x_i^2}{B_i} + \frac{x_i}{B_i c_i}$. Note that this is more efficient than several existing approaches for solving Problem~(P) that claim a linear time complexity (see, e.g., \cite{Ling2012,Khakurel2014}). The reason for this is that these algorithms achieve this complexity only if the gain parameter~$c$ has already been sorted, which is however only the case for some specific communication systems (see, e.g., \cite{Palomar2005}).

Another common objective for the channel power allocation problem~(P) (see, e.g., \cite{Xing2020}) is to minimize the mean square error between different channels from a set~$\mathcal{N}$. This objective is given by
\begin{equation*}
\min_{x \in \mathbb{R}^n} \ \sum_{i \in \mathcal{N}}  \frac{w_i}{A_i x_i + D_i},
\end{equation*}
where $w_i$, $a_i$, and $b_i$ are positive parameters for each $i \in \mathcal{N}$. This objective function is $(a,b,f)$-separable by choosing $a_i := \sqrt{\frac{w_i}{A_i}}$ and $b_i := \frac{D_i}{\sqrt{w_i A_i}}$ for each $i \in \mathcal{N}$ and $f(x_i) := \frac{1}{x_i}$:
\begin{equation*}
a_i f \left(\frac{x_i}{a_i} + b_i \right)
=
\sqrt{\frac{w_i}{A_i}} \frac{1}{x_i \sqrt{\frac{A_i}{w_i}} + \frac{D_i}{\sqrt{w_i A_i}}}
=
\frac{w_i}{\sqrt{A_i}} \frac{1}{x_i \sqrt{A_i} + \frac{D_i}{\sqrt{A_i}}}
=
\frac{w_i}{A_i x_i + D_i}.
\end{equation*}
Moreover, several variations of the channel power allocation problem have been studied with, e.g., bounds on disjoint or nested subsets of allocations (see, e.g., \cite{He2013} and \cite{DAmico2014} respectively). Analogously to Problem~(P), one can show that these problems are instances of $(a,b,f)$-S/GBC/C and $(a,b,f)$-S/NC/C  and thus can be solved as instances of $(a,b)$-Q/GBC/C and $(a,b)$-Q/NC/C respectively.

\subsection{Storage operation in energy systems}
\label{sec_storage}

Storage systems are becoming a crucial part of current and future sustainable energy systems (see, e.g., \cite{Roberts2011, Lund2016,Zame2018}). Such systems support satisfying the energy demand of, e.g., a neighborhood, when renewable energy sources such as solar and wind are insufficient due to, e.g., unfavorable weather conditions. Commonly, the operation of the storage systems is done in a way that the stress on the overall grid is reduced as much as possible. Determining for a given time horizon the best operational schedule for the storage, i.e., how much energy should be (dis)charged at each moment to reach the overall goal in the best way, leads to an optimization problem. In this problem, we divide the overall time horizon into $n$ equidistant time intervals of length $\Delta t$ indexed by the set $\mathcal{N} := \lbrace 1,\ldots,n \rbrace$ and determine for each interval $i \in \mathcal{N}$ the (dis)charging power $x_i$ during this interval. This amount is limited by the minimum and maximum charging rates $X_{\min}$ and $X_{\max}$. Moreover, the charging must be done such that the storage capacity $D$ is not exceeded. Given the initial amount of energy $S_{\text{start}}$ in the storage and a desired target amount $S_{\text{end}}$ at the end of the horizon, the storage operation problem can be formulated as follows (see also \cite{vdKlauw2017}):
\begin{align*}
\text{(B)} \ : \ \min_{x \in \mathbb{R}^n} \ & \sum_{i \in \mathcal{N}} \phi_i (x_i) \\
\text{s.t. } &0 \leq S_{\text{start}} + \Delta t \sum_{i=1}^j x_i \leq D, \quad j \in \mathcal{N} \backslash \lbrace n \rbrace, \\
& S_{\text{start}} + \Delta t \sum_{i \in \mathcal{N}} x_i = S_{\text{end}}, \\
& X_{\min} \leq x_i \leq X_{\max}, \quad i \in \mathcal{N},
\end{align*}
where the functions $\phi_i$ represent the desired grid objective. Note that if each function $\phi_i$ is convex, which is in general the case in this problem setting, this problem is an instance of S/NC/C.

Three commonly seen objectives that are used to reduce grid stress and congestion are: minimal import and export of energy from the main grid (also known as energy-autarky, see, e.g., \cite{Muller2011}), load profile flattening (see, e.g., \cite{Gerards2015}), and minimizing peak consumption (see, e.g., \cite{Uddin2018}). One way to model the latter case is to set a maximum level $M$ for the overall power consumption of the neighborhood. Given the power consumption $p := (p_i)_{i \in \mathcal{N}}$ of the neighborhood, we can model these objectives as follows:
\begin{align*}
\text{Minimizing exchange with main grid: } & \phi_i (x_i) = |x_i + p_i |; \\
\text{Load profile flattening: } & \phi_i (x_i) = (x_i + p_i)^2; \\
\text{Threshold peak shaving: } & \phi_i (x_i) = \begin{cases} 0 & \text{if } x_i + p_i \leq M, \\
\underline{f}(x_i + p_i) & \text{if } x_i + p_i > M, \end{cases}
\end{align*}
where $\underline{f}$ is a convex non-decreasing function with $\underline{f}(M) = 0$. Note that for the objective of load profile flattening, Problem~(B) is an instance of $(\bar{e},p)$-Q/NC/C, where $\bar{e}$ is the vector of ones. Moreover, for the other two objectives, Problem~(B) is an instance of $(\bar{e},p,f)$-S/NC/C where $f$ is the absolute value function or the piecewise function
\begin{equation*}
f (y) = \begin{cases}
0 & \text{if } y \leq M, \\
\underline{f}(y) & \text{if } y > M.
\end{cases}
\end{equation*}
It follows by Corollary~\ref{col_same} that the optimal solution to $(\bar{e},p)$-Q/NC/C is also optimal for $(\bar{e},p,f)$-S/NC/C for these two functions. This implies that we can schedule the storage (dis)charging such that all three objectives are satisfied simultaneously by aiming for load profile flattening. This is an effect that can also be observed for other renewable energy systems such as photovoltaic (solar panel) systems and electric vehicle charging (see, e.g., \cite{Munkhammar2013}) and heat pumps (see, e.g., \cite{Vanhoudt2014}). Moreover, energy tariff systems that employ piecewise linear cost functions have been shown to be able to flatten the load profile, i.e., the objective modeled by a quadratic cost function (see, e.g.,  \cite{Reijnders2020}). Since such tariff systems are simpler to explain to end users, they are more likely to be accepted than systems using quadratic cost functions while still achieving the desired objective of load profile flattening.

\subsection{Stratified sampling}

Stratified sampling is a sampling method suitable for situations where it is likely that a random sample is not a proper representation of the population \cite{Neyman1934}. Such a situation occurs, e.g., when several subclasses of the population score extremely on the to-be-estimated characteristic. To deal with this specific case, we partition the given population into~$n$ so-called \emph{strata} with sizes $N_1,\ldots, N_n$ that, ideally, represent the aforementioned subclasses. Given the desired overall sample size~$R$, the goal is to determine for each stratum~$i \in \mathcal{N} := \lbrace 1,\ldots,n \rbrace$ the number of samples $x_i$ drawn from this stratum while minimizing the variance of the given characteristic. Following the formulation in \cite{Friedrich2015}, the optimal sample allocation is the solution of the following optimization problem:
\begin{align*}
\min_{x \in \mathbb{Z}^n} \ & \sum_{i \in \mathcal{N}} \left(\frac{N_i^2 S_i^2}{x_i} - N_i S_i^2 \right) \\
\text{s.t. } & \sum_{i \in \mathcal{N}} x_i = R, \\
& 0 \leq x_i \leq N_i, \quad i \in \mathcal{N},
\end{align*}
where $S_i^2$ is the variance of the characteristic within stratum~$i$. Similarly to \cite{Friedrich2015}, the sample bounds of $0$ and $N_i$ can be chosen differently to ensure a minimum or maximum number of samples drawn from a given stratum.

Let $D \in \mathbb{R}^n$ be a vector with $D_i := N_i^2 S_i^2$ for all $i \in \mathcal{N}$. Then the above problem is an instance of the problem $(D,0,f)$-S/Box/I with $f(x_i) = \frac{1}{x_i}$. Thus, by Corollary~\ref{col_box_complexity}, we can solve this problem as an instance of $(D,0)$-Q/Box/I in $O(n)$ time. Note that, in contrast to the approaches in, e.g., \cite{Friedrich2015}, this complexity depends only on the number~$n$ of strata and not on the actual strata sizes $N_1,\ldots,N_n$ or desired sample size~$R$. As a consequence, our reduction result yields a promising approach to determine optimal sample sizes in large datasets, which can contain billions of samples (see, e.g., \cite{Meng2013}).

\subsection{Vessel speed optimization}

A recent trend in ship routing is to actively manage the ship's sailing speed to reduce fuel costs and carbon emissions \cite{Psaraftis2014}. As a consequence, when determining the routes of a fleet of ships to deliver cargo within given timing constraints, one must be able to determine the minimum cost of having a ship sail a given route. This problem is known as the vessel speed optimization problem (see, e.g., \cite{Norstad2011,Hvattum2013}). In this problem, we are given a route between $n+1$ ports starting at port~0 at time $t^{\text{start}}$ and required to finish at time $t^{\text{end}}$ at port~$n$. The distance between consecutive ports~$i-1$ and~$i$ is given by $d_i$ and each port~$i$ must be serviced by the ship within a given time window $[A_i, D_i]$. The goal is to determine for each leg~$i \in \mathcal{N} := \lbrace 1,\ldots,n \rbrace$ of the tour, i.e., for each distance $d_i$, a speed $v_i$ such that the fuel cost of sailing at these speeds is minimized. Following \cite{Norstad2011,Hvattum2013}, we formulate this problem as follows:
\begin{align*}
\text{(V)} \ : \ \min_{t \in \mathbb{R}^{n+1}, \ v \in \mathbb{R}^n} \ &
\sum_{i \in \mathcal{N}} d_i c(v_i) \\
\text{s.t. }  & t_i + \frac{d_i}{v_i} = t_{i+1}, \quad i \in \lbrace 0 \rbrace \cup \mathcal{N} \backslash \lbrace n \rbrace, \\
& A_i \leq t_i \leq D_i, \quad i \in \mathcal{N} \backslash \lbrace n \rbrace, \\
& t_0 = t^{\text{start}}, \ t_n = t^{\text{end}}, \\
& v^{\min} \leq v_i \leq v^{\max}, \quad i \in  \mathcal{N}.
\end{align*}
Here, $v^{\min}$ and $v^{\max}$ are the minimum and maximum cruising speeds and $c$ is a non-decreasing convex function that models the relation between sailing speed and fuel costs per unit distance.

From this formulation, it follows by induction on $i$ that $t_i = t^{\text{start}} + \sum_{k=1}^i \frac{d_k}{v_k}$ for all $i \in \mathcal{N}$ and that $t^{\text{start}} + \sum_{k=1}^n \frac{d_k}{v_k} = t^{\text{end}}$. Let $x_i := \frac{d_i}{v_i}$ and $q(x_i) := c(1/x_i)$. It follows that
\begin{equation*}
d_i c(v_i) = d_i c \left( \frac{d_i}{x_i} \right)
= d_i q \left( \frac{x_i}{d_i} \right).
\end{equation*}
Note that $q$ is convex since $c$ is non-decreasing. This means that Problem~(V) is equivalent to the following convex optimization problem:
\begin{align*}
\min_{x \in \mathbb{R}^n} \ & \sum_{i \in \mathcal{N}} d_i q \left( \frac{x_i}{d_i} \right) \\
\text{s.t. } & A_i - t^{\text{start}} \leq \sum_{k=1}^i x_i \leq D_i - t^{\text{start}}, \quad i \in \mathcal{N} \backslash \lbrace n \rbrace, \\
& \sum_{i \in \mathcal{N}} x_i = t^{\text{end}} - t^{\text{start}}, \\
& \frac{d_i}{v^{\max}} \leq x_i \leq \frac{d_i}{v^{\min}}, \quad i \in \mathcal{N}.
\end{align*}
This problem is an instance of $(d,0,q)$-S/NC/C. Hence, by Corollary~\ref{col_NC}, this problem and thus Problem~(V) can be solved in $O(n \log n)$ time by, e.g., the fast algorithm in \cite{SchootUiterkamp2020b}. This result is relevant since Problem~(V) often occurs as a subproblem in fleet routing algorithms \cite{Psaraftis2014} and thus using a faster algorithm for this subproblem can lead to significant speed-ups for the overall algorithm.

\subsection{Speed scaling}

Efficient energy usage is an important topic within the development of computing systems \cite{Zhuravlev2013}. To reduce energy consumption, modern computer processors can adjust their speed to save energy while still meeting their performance constraints. This leads to scheduling problems where a set of tasks needs to be scheduled and processor speeds need to be chosen such that all tasks are executed before their deadline (see \cite{Gerards2016aa} for a survey). One special case of these types of scheduling problems is the case where the deadlines are agreeable, i.e., deadlines are ordered according to the arrival times of the tasks (see also \cite{Bampis2012}). In this problem, we are given~$n$ tasks indexed by the set $\mathcal{N}$ that must be processed on a single processor. Each task~$i \in \mathcal{N}$ has an arrival time $A_i$, deadline $D_i$, and amount of work $w_i$ that can be interpreted as the amount of operations and calculations the processor must execute to perform this task. The goal is to select for each task~$i$ an execution speed~$s_i$ and starting time~$B_i$ such that each task is processed before its deadline and the total energy usage of the processor is minimized.

Since the deadlines are agreeable we have that $D_i \leq D_k$ if $A_i \leq A_k$ for any two tasks~$i,k \in \mathcal{N}$. Moreover, in an optimal schedule, the tasks can be scheduled in non-decreasing order of their deadlines \cite{Bampis2012}. This means that we can formulate this speed scaling problem as follows (see also \cite[Chapter~4]{Gerards2014}:
\begin{align*}
\text{(S)} \ : \ \min_{s \in \mathbb{R}^n \ B \in \mathbb{R}^n} \ &
\sum_{i \in \mathcal{N}} p(s_i) \frac{w_i}{s_i} \\
\text{s.t. } & B_i + \frac{w_i}{s_i} \leq D_i, \quad i \in \mathcal{N}, \\
& B_i \geq A_i, \quad i \in \lbrace 1,\ldots,n \rbrace, \\
& B_i + \frac{w_i}{s_i} \leq B_{i+1}, \quad i \in \mathcal{N} \backslash \lbrace n \rbrace, \\
& 0 < s_i \leq s^{\max}, \quad i \in \mathcal{N},
\end{align*}
where $s^{\max}$ is the maximum processor speed and $p$ is a convex function that models the relation between processor speed and its energy usage. Note that we can impose a nonzero lower bound on each $s_i$ so that the feasible set of this problem is guaranteed to be closed. Since we must choose the speeds such that each task can be executed in the maximum time that is available for it, we have that $\frac{w_i}{s_i} \leq D_i - A_i$. This yields a lower bound on $s_i$ of $\frac{w_i}{D_i - A_i}$ that is nonzero since $w_i > 0$ and $D_i > A_i$. 

If the processor is active until the latest deadline regardless of the scheduling of the tasks, then there exists an optimal schedule with no idle time \cite{Irani2007}. This means that we can add without loss of generality the constraint $B_i = \sum_{k=1}^i \frac{w_i}{s_i}$ for all $i \in \mathcal{N}$ to the formulation of Problem~(S). Let $x_i := \frac{w_i}{s_i}$ for all $i \in \mathcal{N}$ and $q(x_i) := x_i p(1/x_i)$ (note that $q$ is convex). It follows that
\begin{equation*}
p(s_i) \frac{w_i}{s_i} = p\left(\frac{w_i}{x_i} \right) x_i = w_i q\left(\frac{x_i}{w_i} \right).
\end{equation*}
Using the lower bound on $s_i$, the added constraint on~$B_i$, the transformation $x_i = \frac{w_i}{s_i}$, and the function~$q$, we can reformulate Problem~(S) to
\begin{align*}
\min_{x \in \mathbb{R}^n} \ &
\sum_{i \in \mathcal{N}} w_i q \left(\frac{x_i}{w_i} \right) \\
\text{s.t. } & A_{i+1} \leq \sum_{k=1}^i x_k \leq D_i, \quad i \in \mathcal{N} \backslash \lbrace n \rbrace, \\
& \sum_{i \in \mathcal{N}} x_i = D_n, \\
& \frac{w_i}{s^{\max}} \leq x_i \leq D_i - A_i, \quad i \in \mathcal{N}.
\end{align*}
This is an instance of $(w,0,q)$-S/NC/S. Hence, by Corollary~\ref{col_NC}, this problem and thus Problem~(S) can be solved in $O(n \log n)$ time. This result also leads to complexity improvements for speed scaling problems that can be reduced to Problem~(S), e.g., for the multi-core processor scheduling problem considered in \cite{Gerards2014a}.

Recently, \cite{Shioura2017} applied the equivalence result in \cite{Nagano2012} to improve the time complexity of several other speed scaling problems. Together with the result in this section, this suggests that there is a great potential for using the reduction result in this article to contribute to more efficient algorithms within this research field.

\section{Conclusions and outlook}
\label{sec_concl}

In this article, we studied the resource allocation problem (RAP) with additional submodular constraints. We proved that the class of RAPs whose objective function is $(a,b,f)$-separable can be solved efficiently as quadratic RAPs if a certain optimality condition of the general separable problem is satisfied. Using this reduction result, we derive new worst-case time complexity results on several relevant special cases of the studied problem. Moreover, we have shown the impact of our reduction result on several core problems in wireless communications, smart grids, statistics, routing, and processor management.

One major direction for future research is the extension of the reduction result to other problems. The most intuitive starting point for this is to search for other optimization problems that satisfy the required optimality condition. Promising candidates for this are problems that are variations on the RAPs studied in this article, e.g., RAPs with interval and cardinality constraints \cite{Sun2013,SchootUiterkamp2018b} and with additional nonseparable terms in the objective functions \cite{Romeijn2007,Sharkey2011,SchootUiterkamp2019a}.
Besides this more technical direction, in the light of a more cross-disciplinary approach towards the study of RAPs, it is worthwhile to identify more research fields and applications, next to the ones that we discussed in this article, where RAPs are being studied and where our results can have impact and lead to new insights.

\section*{Acknowledgments}
This research has been conducted within the SIMPS project (647.002.003) supported by NWO and Eneco.

\appendix

\section{Laminar constraints are a special case of submodular constraints}
\label{app_lemma_LC}
In this appendix, we show that laminar (or tree) constraints are a special case of submodular constraints. Recall that
\begin{itemize}
\item laminar constraints are of the form $L_j \leq \sum_{i \in \mathcal{N}_j} x_i \leq U_j$, $j \in \mathcal{M}$, where the subsets $\mathcal{N}_1,\ldots,\mathcal{N}_m$ of $\mathcal{N}$ have the property that either $\mathcal{N}_j \cap \mathcal{N}_{\ell} = \emptyset$, $\mathcal{N}_j \subset \mathcal{N}_{\ell}$, or $\mathcal{N}_j \supset \mathcal{N}_{\ell}$ for all $j,\ell \in \mathcal{M}$;
\item
submodular constraints are of the form $\sum_{i \in \mathcal{S}} x_i \leq r(\mathcal{S})$, $\mathcal{S} \subset \mathcal{N}$ and $\sum_{i \in \mathcal{N}} x_i = r(\mathcal{N})$, where $r$ is a submodular function.
\end{itemize}
For this, we use a result from \cite{Fujishige1984, Fujishige2005} on so-called \emph{cross-free} families of subsets. A family $\mathcal{F} \subseteq 2^{\mathcal{N}}$ is called cross-free if none of its elements cross, i.e., for any two subsets $\mathcal{X}, \mathcal{Y} \in \mathcal{F}$ we have that at least one of the sets $\mathcal{X} \cap \mathcal{Y}$, $\mathcal{X} \cap (\mathcal{N} \backslash \mathcal{Y})$, $(\mathcal{N} \backslash \mathcal{X}) \cap \mathcal{Y}$, or $(\mathcal{N} \backslash \mathcal{X}) \cap (\mathcal{N} \backslash \mathcal{Y})$ is empty. For a given cross-free family $\mathcal{F}$ containing $\emptyset$ and $\mathcal{N}$ and for any set function $r : \ \mathcal{F} \rightarrow \mathcal{R}$ with $r(\emptyset) = 0$, the set
\begin{equation*}
\mathcal{B}(\mathcal{F},r) := \left\{ x \in \mathbb{R}^n \ \middle| \ \sum_{i \in \mathcal{X}} x_i \leq r(\mathcal{X}) \ \forall \mathcal{X} \in \mathcal{F}, \ \sum_{i \in \mathcal{N}} x_i = r(\mathcal{N}) \right\}
\end{equation*}
is a base polyhedron \cite{Fujishige1984, Fujishige2005}. This means that there exists a submodular function $r'  : \ 2^{\mathcal{N}} \rightarrow \mathbb{R}$ such that 
\begin{equation*}
\mathcal{B}(\mathcal{F},r) = \left\{ x \in \mathbb{R}^n \ \middle| \ \sum_{i \in \mathcal{X}} x_i \leq r'(\mathcal{X}) \ \forall \mathcal{X} \subset \mathcal{N}, \ \sum_{i \in \mathcal{N}} x_i = r(\mathcal{N}) \right\} .
\end{equation*}
Thus, we can show that laminar constraints are a special case of submodular constraints if for a given feasible set $\mathcal{C}'$ determined by laminar constraints we can find a cross-free family $\mathcal{F}$ and a set function $r :  \mathcal{F} \rightarrow \mathbb{R}$ such that $\mathcal{B}(\mathcal{F},r) = \mathcal{C}'$.

For given laminar constraints $L_j \leq \sum_{i \in \mathcal{N}_j} x_i \leq U_j$, $j \in \mathcal{M}$ and a feasible set $\mathcal{C}' := \lbrace x \in \mathbb{R}^n \ | \ L_j \leq \sum_{i \in \mathcal{N}_j} \leq U_j$, $j \in \mathcal{M} \rbrace$, we define the following family of subsets of $\mathcal{N}$:
\begin{equation*}
\mathcal{N}' := \lbrace \mathcal{N}_j \ | \ j \in \mathcal{M} \rbrace \cup \lbrace \mathcal{N} \backslash \mathcal{N}_j \ | \ j \in \mathcal{M} \rbrace .
\end{equation*}
Note, that the feasible set $\mathcal{C}'$ is equal to $\mathcal{B}(\mathcal{N}',r')$, where $r' : \ \mathcal{N}' \rightarrow \mathbb{R}$ is a set function on $\mathcal{N}'$ given by
\begin{equation*}
r'(\mathcal{X}) := \begin{cases}
U_j & \text{if } X = \mathcal{N}_j \text{ for some } j \in \mathcal{M}, \\
R - L_j & \text{if } X = \mathcal{N} \backslash \mathcal{N}_j \text{ for some } j \in \mathcal{M}.
\end{cases}
\end{equation*}
We claim that $\mathcal{N}'$ is a cross-free family, which immediately implies that the set $\mathcal{B}(\mathcal{N}',r')$ is a base polyhedron and thus that laminar constraints are a special case of submodular constraints. For this, we consider for two different sets $\mathcal{X}, \mathcal{Y} \in \mathcal{N}'$ four cases:
\begin{enumerate}
\item
If $\mathcal{X} = \mathcal{N}_j$ and $\mathcal{Y} = \mathcal{N}_{\ell}$ for some $j, \ell \in \mathcal{M}$, then either $\mathcal{X} \cap \mathcal{Y}= \emptyset$, $\mathcal{X} \subset \mathcal{Y}$, or $\mathcal{X} \supset \mathcal{Y}$. The latter two cases imply that $\mathcal{X} \cap (\mathcal{N} \backslash \mathcal{Y}) = \empty$ and $(\mathcal{N} \backslash \mathcal{X}) \cap \mathcal{Y} = \emptyset$ respectively. Thus, in all three cases, $\mathcal{X}$ and $\mathcal{Y}$ do not cross.
\item
If $\mathcal{X} = \mathcal{N}_j$ and $\mathcal{Y} = \mathcal{N} \backslash \mathcal{N}_{\ell}$ for some $j, \ell \in \mathcal{M}$, then either $\mathcal{X} \cap (\mathcal{N} \backslash \mathcal{Y})= \emptyset$, $\mathcal{X} \subset (\mathcal{N} \backslash \mathcal{Y})$, or $\mathcal{X} \supset (\mathcal{N} \backslash \mathcal{Y})$. The latter two cases imply that $\mathcal{X} \cap \mathcal{Y} = \emptyset$ and $(\mathcal{N} \backslash \mathcal{X}) \cap (\mathcal{N} \backslash \mathcal{Y}) = \emptyset$ respectively. Thus, in all three cases, $\mathcal{X}$ and $\mathcal{Y}$ do not cross.

\item
If $\mathcal{X} = \mathcal{N} \backslash \mathcal{N}_j$ and $\mathcal{Y} = \mathcal{N}_{\ell}$ for some $j, \ell \in \mathcal{M}$, we can use the argument in case~2 with the roles of $\mathcal{X}$ and $\mathcal{Y}$ interchanged to conclude that $\mathcal{X}$ and $\mathcal{Y}$ do not cross.
\item
If $\mathcal{X} = \mathcal{N} \backslash \mathcal{N}_j$ and $\mathcal{Y} = \mathcal{N} \backslash \mathcal{N}_{\ell}$ for some $j, \ell \in \mathcal{M}$, then either $(\mathcal{N} \backslash \mathcal{X}) \cap (\mathcal{N} \backslash \mathcal{Y})= \emptyset$, $(\mathcal{N} \backslash \mathcal{X}) \subset (\mathcal{N} \backslash \mathcal{Y})$, or $(\mathcal{N} \backslash \mathcal{X}) \supset (\mathcal{N} \backslash \mathcal{Y})$. The latter two cases imply that $(\mathcal{N} \backslash \mathcal{X}) \cap \mathcal{Y} = \emptyset$ and $\mathcal{X} \cap (\mathcal{N} \backslash \mathcal{Y}) = \emptyset$ respectively. Thus, in all three cases, $\mathcal{X}$ and $\mathcal{Y}$ do not cross.
\end{enumerate}

\section{An alternative proof that Condition~\ref{prop_opt_cond} holds for the resource allocation problem with laminar constraints}
\label{sec_LC}
Here we present an alternative proof of the claim that Condition~\ref{prop_opt_cond} holds for the resource allocation problem with laminar constraints (S/LC/$\gamma$). Before we prove this result in Lemma~\ref{lemma_opt_laminar}, we first show that the difference between any two feasible solutions $x$ and $z$ of S/LC/$\gamma$ can be written as a nonnegative combination of vectors in $\mathcal{E}_{\mathcal{C}'}(x)$, where $\mathcal{C}'$ is the feasible set of S/LC/$\gamma$. In other words, $z-x$ belongs to the cone generated by the vectors in $\mathcal{E}_{\mathcal{C}'}(x)$. To this end, we present the following procedure to obtain this combination. Starting from the solution $\bar{x}^0 := x$, we construct a series of intermediate vectors $(\bar{x}^t)_{t \geq 0}$ that finally leads to $z$ by iteratively transferring amounts between two variables. We do this in such a way that the distance $\sum_{i \in \mathcal{N}} |z_i - \bar{x}^t_i|$ reduces as $t$ increases and becomes zero for some $\bar{t} \geq 0$, meaning that $\bar{x}^{\bar{t}} = z$. To ensure finiteness of this process, we always choose two variables with indices $i,k$ such that $\bar{x}^t_i > z_i$ and $\bar{x}^t_k < z_k$. By transferring an amount of $\lambda_{ik} := \min(\bar{x}^t_i - z_i, z_k - \bar{x}^t_k)$ between those variables, we have for the subsequent vector $\bar{x}^{t+1}$ that either $\bar{x}^{t+1}_i = z_i$ or $\bar{x}^{t+1}_k = z_k$. By repeating this process, we finally reach an intermediate vector $\bar{x}^{\bar{t}}$ that equals $z$. For each selected pair $(i,k)$, the value $\lambda_{ik}$ represents a positive coefficient in the desired conic combination.

To ensure that each index pair with a positive coefficient is an exchangeable pair (see also Lemma~\ref{lemma_member}), i.e., is in $\mathcal{E}_{\mathcal{C}'}(x)$, we restrict the choice of index pair in the procedure in the following way. First, we order the subsets such that $\mathcal{N}_{j} \subset \mathcal{N}_{j'}$ implies $j > j'$ for all $j,j' \in \mathcal{M}$. Moreover, we define $\mathcal{N}_0 := \mathcal{N}$. Now we iterate through the subsets from $\mathcal{N}_m$ to $\mathcal{N}_0$ and during iteration $j$ we allow only exchanges between variables whose indices belong to the current subset $\mathcal{N}_j$.

The procedure is summarized in Algorithm~\ref{alg_combi}. In this algorithm, for any $j \in \lbrace 0 \rbrace \cup \mathcal{M}$, $t_j$ is the last iteration index such that no exchanges are allowed between a variable whose index is in $\mathcal{N}_j$ and a variable whose index is not in $\mathcal{N}_j$. 
\begin{algorithm}
\caption{Computing $z-x$ as a conic combination of vectors in $\mathcal{E}_{\mathcal{C}'}(x)$.}
\label{alg_combi}
\begin{algorithmic}[1]
\STATE{\textbf{Input:} Two feasible solutions $x$, $z$ to S/LC/$\gamma$}
\STATE{\textbf{Output:} Weight matrix $\lambda \in \mathbb{R}^{n \times n}_{\geq 0}$}
\STATE{Initialize $\lambda_{ik} = 0$ for all $i,k \in \mathcal{N}$}
\STATE{Order subsets such that $\mathcal{N}_{j} \subset \mathcal{N}_{j'}$ implies $j > j'$ for all $j,j' \in \mathcal{M}$}
\STATE{$\mathcal{N}_0 := \mathcal{N}$; $t = 0$; $\bar{x}^{0} := x$}
\FOR{$j=m$ down to $0$}
\WHILE{there exist $i,k \in \mathcal{N}_j$ such that $\bar{x}_i^{t} > z_i$ and $\bar{x}_k^{t} < z_k$}
\STATE{$\lambda_{ik} := \min (\bar{x}_i^{t} - z_i, z_k - \bar{x}_k^{t})$}
\STATE{$\bar{x}^{t+1} := \bar{x}^{t} + \lambda_{ik}(e^k - e^i)$}
\STATE{$t = t + 1$}
\ENDWHILE
\STATE{$t_j = t$}
\ENDFOR
\end{algorithmic}
\end{algorithm}

In Lemma~\ref{lemma_combi}, we prove several properties of the output $\lambda$ of the algorithm and of the intermediate vectors $(\bar{x}^t)_{t \geq 0}$.
\begin{lemma}
The following statements hold for the output $\lambda$ and the sequence of intermediate vectors $(\bar{x}^t)_{t \geq 0}$ of Algorithm~\ref{alg_combi} when applied to two feasible solutions $x$ and $z$ to S/LC/$\gamma$:
\begin{enumerate}
\item
$\sum_{i \in \mathcal{N}} \bar{x}_i^{t} = C$ for all $t \geq 0$.
\item
If $x_i^{t} > z_i$ for a given $t \geq 0$, then $z_i \leq \bar{x}_i^{t'} \leq \bar{x}_i^{t} \leq x_i$ for all $t' > t$;
\item
If $x_i^{t} < z_i$ for a given $t \geq 0$, then $z_i \geq \bar{x}_i^{t'} \geq \bar{x}_i^{t} \geq x_i$ for all $t' > t$;
\item
If $x_i^{t} = z_i$ for a given $t \geq 0$, then $x_i^{t'} = z_i$ for all $t' > t$.
\item
For a given $j$ and $t \geq t_j$, we have that either $\bar{x}_i^{t} \leq z_i$ for all $i \in \mathcal{N}_j$ or $\bar{x}_i^t \geq z_i$ for all $i \in \mathcal{N}_j$.
\item
Each index pair $(i,k) \in \mathcal{N}^2$ is selected at most once over the entire course of the algorithm.
\item
For a given $j$ and any $t \leq t_j$, it holds that $\sum_{\ell \in \mathcal{N}_j} \bar{x}^{t}_{\ell} = \sum_{\ell \in \mathcal{N}_j} x_{\ell}$.
\item
$z-x = \sum_{(i,k) \in \mathcal{N}^2} \lambda_{ik}(e^k - e^i) $.
\end{enumerate}
\label{lemma_combi}
\end{lemma}
\begin{proof}
Part~(1): Follows by induction on $t$ since $\sum_{\ell \in \mathcal{N}} \bar{x}_{\ell}^{t+1} = \sum_{\ell \in \mathcal{N}} \bar{x}_{\ell}^{t} + \lambda_{ik} - \lambda_{ik} = \sum_{\ell \in \mathcal{N}} \bar{x}_{\ell}^{t}$ for all $t \geq 0$ and $\bar{x}^{0} = x$.

Part~(2): For a given $t \geq 0$, we have that $\bar{x}^{t}_i > z_i$ implies that either $\bar{x}^{t+1}_i = \bar{x}^{t}_i$ (if $i$ is not selected during iteration~$t$) or $z_i \leq \bar{x}^{t+1}_i < \bar{x}^{t}_i$ (if $i$ is selected during iteration~$t$). Thus, we have that $\bar{x}^{t}_i > z_i$ implies that $z_i \leq \bar{x}^{t+1}_i \leq \bar{x}^{t}_i$. By induction, one can deduce that if $x_i^t > z_i$, then $z_i \leq \bar{x}^{t'}_i \leq \bar{x}^{t}_i \leq x_i$ for all $t \geq 0$ and $t' > t$.

Part~(3): Is analogous to the proof of Part~(2).

Part~(4): If $x_i^{t} = z_i$, then $i$ will not be selected anymore as part of an exchangeable pair. Hence, $x_i^{t} = x_i^{t+1} = \dots = z_i$.

Part~(5): By definition of $t_j$, we have that either $\bar{x}^{t_j}_{\ell} \geq z_{\ell}$ for all $\ell \in \mathcal{N}_j$ or $\bar{x}^{t_j}_{\ell} \leq z_{\ell}$ for all $\ell \in \mathcal{N}_j$. It follows directly from Parts~(2)-(4) that in the first case $\bar{x}^{t}_{\ell} \geq z_{\ell}$ for all $\ell \in \mathcal{N}_j$ and that in the second case $\bar{x}^{t}_{\ell} \leq z_{\ell}$ for all $\ell \in \mathcal{N}_j$.

Part~(6): If the pair $(i,k)$ is chosen during some iteration~$t$, then either $\bar{x}^{t+1}_i = z_i$ or $\bar{x}^{t+1}_k = z_k$. Thus, at least one of the indices $i,k$ cannot be chosen again as part of a pair, hence the pair $(i,k)$ is selected at most once.

Part~(7): For a given $t \leq t_j$, let $(i,k)$ denote the selected pair during iteration~$t-1$. Thus, there is a subset $\mathcal{N}_{j'}$ with $j' > j$ such that $i,k \in \mathcal{N}_{j'}$. By the ordering of the subsets, we have either $\mathcal{N}_{j} \cap \mathcal{N}_{j'} = \emptyset$ or $\mathcal{N}_{j'} \subset \mathcal{N}_{j}$. Thus, either both or neither of the indices $i$ and $k$ are in $\mathcal{N}_j$. This implies that $\sum_{\ell \in \mathcal{N}_j} \bar{x}_{\ell}^{t} = \sum_{\ell \in \mathcal{N}_j} \bar{x}_{\ell}^{t-1}$. By induction on $t$, it follows that $\sum_{\ell \in \mathcal{N}_j} \bar{x}_{\ell}^{t} = \sum_{\ell \in \mathcal{N}_j} \bar{x}_{\ell}^{0} =  \sum_{\ell \in \mathcal{N}_j} x_{\ell}$.

Part~(8): Follows from Part~(6) and the fact that $\lambda_{ik} = 0$ if the pair $(i,k)$ has not been chosen during any iteration.
\end{proof}

Lemma~\ref{lemma_combi} implies that for any two feasible solutions $x$ and $z$, the difference $z-x$ can be written as a nonnegative combination of the vectors $(e^k - e^i)_{(i,k) \in \mathcal{N}^2}$. We strengthen this result in Lemma~\ref{lemma_member} by proving that $z-x$ can be written as a nonnegative combination of the vectors in $\mathcal{E}_{\mathcal{C}'}(x)$.
\begin{lemma}
Let $\lambda$ and $(\bar{x}^t)_{t \geq 0}$ be the output of Algorithm~\ref{alg_combi} applied to two feasible solutions $x$ and $z$ of the problem S/LC/$\gamma$. If $\lambda_{ik} > 0$ for a given pair $(i,k) \in \mathcal{N}^2$, then $(i,k) \in \mathcal{E}_{\mathcal{C}'}(x)$ and $\lambda_{\ell,i} = \lambda_{k,\ell} = 0$ for all $\ell \in \mathcal{N}$.
\label{lemma_member}
\end{lemma}
\begin{proof}
Note that for any two indices $i,k \in \mathcal{N}$, the solution $x + \epsilon (e^k - e^i)$ is feasible for some $\epsilon > 0$ if and only if we have for each subset $\mathcal{N}_j$ that contains $i$ but not $k$ that $\sum_{\ell \in \mathcal{N}_j} x_{\ell} > L_j$, and for each subset $\mathcal{N}_{j'}$ that contains $k$ but not $i$ that $\sum_{\ell \in \mathcal{N}_{j'}} x_{\ell} < U_{j'}$. Let $\mathcal{N}_{j'}$ be the minimal subset that contains both $i$ and $k$, i.e., there is no other subset $\mathcal{N}_j$ such that $\mathcal{N}_j \subset \mathcal{N}_{j'}$ and $i,k \in \mathcal{N}_j$. If $\lambda_{ik} > 0$, then there exists $t_{j'+1} < t \leq t_{j'}$ such that the pair $(i,k)$ has been selected during iteration~$t$. Thus, $\bar{x}_i^t > \bar{x}_i^{t+1} \geq z_i$ and $\bar{x}_k^t < \bar{x}_k^{t+1} \leq z_k$. By Parts~(2) and~(3) of Lemma~\ref{lemma_combi}, this means that $x_i > z_i$ and $x_k < z_k$ and that $\bar{x}^{t}_i \geq z_i$ and $\bar{x}^{t}_k \leq z_k$ for all $t \geq 0$. By Part~(5) of Lemma~\ref{lemma_combi}, this means that for any subset $\mathcal{N}_j$ that contains $i$ but not $k$ we have that $\bar{x}^{t_j}_{\ell} \geq z_{\ell}$ for all $\ell \in \mathcal{N}_j$ since $j > j'$. In particular, we have by Part~(2) that $\bar{x}^{t_j}_i > z_i$ since $\bar{x}^{t}_i > z_i$. It follows from feasibility of $z$ and Part~(7) that $L_j \leq \sum_{\ell \in \mathcal{N}_j} z_{\ell} < \sum_{\ell \in \mathcal{N}_j} \bar{x}^{t_j}_{\ell} = \sum_{\ell \in \mathcal{N}_j} x_{\ell}$. Analogously, we can show that $U_j > \sum_{\ell \in \mathcal{N}_j} x_{\ell}$. Thus, the solution $x + \epsilon(e^k - e^i)$ is feasible for some $\epsilon > 0$, hence $(i,k) \in \mathcal{E}_{\mathcal{C}'}(x)$.

Note that for any $\ell \in \mathcal{N}$, we can only have that $\lambda_{\ell i} >0$ if there is some iteration~$t$ with $\bar{x}^{t}_i < z_i$. Since $\bar{x}^{t}_i \geq z_i$ for all $t \geq 0$, we must have that $\lambda_{\ell i} = 0$. Analogously, we must have that $\lambda_{k \ell} = 0$ since $\bar{x}^{t}_k \leq z_k$ for all $t \geq 0$.
\end{proof}

Lemma~\ref{lemma_member} implies that we can partition $\mathcal{N}$ into three subsets such that one subset contains all indices~$i$ for which $\lambda_{ik} > 0$ for at least one $k \in \mathcal{N}$, one subset contains all indices $i$ for which $\lambda_{ki} > 0$ for at least one $k \in \mathcal{N}$, and one subset contains all indices $i$ such that $\lambda_{ik} = \lambda_{ki} = 0$ for all $k \in \mathcal{N}$. More precisely, we can define the following partition of $\mathcal{N}$:
\begin{align*}
\mathcal{L}(x) &:= \lbrace i \in \mathcal{N} \ | \ \lambda_{ik} > 0 \text{ for some } k \in \mathcal{N} \rbrace, \\
\mathcal{U}(x) &:= \lbrace i \in \mathcal{N} \ | \ \lambda_{ki} > 0 \text{ for some } k \in \mathcal{N} \rbrace, \\
\mathcal{F}(x) &:= \mathcal{N} \backslash (\mathcal{L}(x) \cup \mathcal{U}(x))
= \lbrace i \in \mathcal{N} \ | \ \lambda_{ik} = \lambda_{ki} = 0 \text{ for all } k \in \mathcal{N} \rbrace.
\end{align*}
Using this partition and Lemma~\ref{lemma_member}, we can show that S/LC/$\gamma$ satisfies Condition~\ref{prop_opt_cond}.
\begin{lemma}
For $\gamma \in \lbrace \text{C},\text{I} \rbrace$, the problem S/LC/$\gamma$ satisfies Condition~\ref{prop_opt_cond}.
\label{lemma_opt_laminar}
\end{lemma}
\begin{proof}
First, we prove the ``only if''-part. Suppose $x$ is optimal for S/LC/$\gamma$ and there exists an index pair $(i,k) \in \mathcal{E}_{\mathcal{C}'}(x)$ such that $\phi_k^{+}(x_k) < \phi_i^{-}(x_i)$. By definition of $\mathcal{E}_{\mathcal{C}'}(x)$ and the left and right derivatives $\phi_k^{+}$ and $\phi_i^{-}$, there exists $\epsilon > 0$ such that $x + \epsilon(e^k - e^i)$ is feasible and
\begin{equation*}
\phi_k (x_k + \epsilon) + \phi_i (x_i - \epsilon) < \phi_k (x_k) + \phi_i (x_i).
\end{equation*}
This implies that the objective value of $x + \epsilon(e^k-e^i)$ is smaller than that of $x$. Hence, $x$ cannot be optimal, which is a contradiction. It follows that $\phi_k^{+}(x_k) \geq \phi_i^{-}(x_i)$ for all $(i,k) \in \mathcal{E}_{\mathcal{C}'}(x)$.

Second, we prove the ``if''-part. Let $x$ be a feasible solution such that $\phi_k^{+}(x_k) \geq \phi_i^{-}(x_i)$ for all $(i,k) \in \mathcal{E}_{\mathcal{C}'}(x)$ and let $z$ be an arbitrary feasible solution. Moreover, let $\lambda \in \mathbb{R}^{n \times n}$ denote the output of Algorithm~\ref{alg_combi} when applied to $x$ and $z$. By Lemma~\ref{lemma_member} and definition of the sets $\mathcal{L}(x)$, $\mathcal{U}(x)$, and $\mathcal{F}(x)$, we have that
\begin{align*}
z-x &= \sum_{(i,k) \in \mathcal{N}^2} \lambda_{ik}(e^k - e^i)
= \sum_{(i,k) \in \mathcal{E}_{\mathcal{C}'}(x)} \lambda_{ik}(e^k - e^i) \\
&= \sum_{\substack{(i,k) \in \mathcal{E}_{\mathcal{C}'}(x),\\ i \in \mathcal{L}(x)}} \lambda_{ik}(e^k - e^i) 
= \sum_{\substack{(i,k) \in \mathcal{E}_{\mathcal{C}'}(x),\\ i \in \mathcal{L}(x), \\ k \in \mathcal{U}(x)}} \lambda_{ik}(e^k - e^i).
\end{align*}
We define the following subgradient $g \in \mathbb{R}^n$ at the solution $x$:
\begin{equation*}
g_i \begin{cases}
:= \phi_i^{-}(x_i) & \text{if } i \in \mathcal{L}(x), \\
:= \phi_i^{+}(x_i) & \text{if } i \in \mathcal{U}(x), \\
\in [\phi_i^{-}(x_i),\phi_i^{+}(x_i)] & \text{if } i \in \mathcal{F}(x).
\end{cases}
\end{equation*}
By convexity of the functions $\phi_i$, it follows that
\begin{align*}
\sum_{i \in \mathcal{N}} (\phi_i(z_i) - \phi_i(x_i))
&\geq
g^{\top} (z - x)
=
\sum_{\substack{(i,k) \in \mathcal{E}_{\mathcal{C}'}(x),\\ i \in \mathcal{L}(x), \\ k \in \mathcal{U}(x)}} \lambda_{ik} g^{\top}(e^k - e^i)  \\
&= \sum_{\substack{(i,k) \in \mathcal{E}_{\mathcal{C}'}(x),\\ i \in \mathcal{L}(x), \\ k \in \mathcal{U}(x)}} \lambda_{ik} g^{\top}(e^k - e^i) 
= 
\sum_{\substack{(i,k) \in \mathcal{E}_{\mathcal{C}'}(x),\\ i \in \mathcal{L}(x), \\ k \in \mathcal{U}(x)}} \lambda_{ik} (\phi_k^{+}(x_k) - \phi_i^{-}(x_i))
 \geq 0.
\end{align*}
It follows that $x$ is optimal since $z$ is an arbitrary feasible solution.
\end{proof}

\bibliographystyle{abbrv}

\bibliography{Reduction_library}

\begin{thebibliography}{10}

\bibitem{Akhil2018}
P.~T. Akhil and R.~Sundaresan.
\newblock {Algorithms for separable convex optimization with linear ascending
  constraints}.
\newblock {\em Sādhanā}, 43(9):146, 2018.

\bibitem{Alexandrescu2017}
A.~Alexandrescu.
\newblock {Fast deterministic selection}.
\newblock In C.~S.~I. Raman, S.~P. Pissis, S.~J. Puglisi, and Rajeev, editors,
  {\em Leibniz International Proceedings in Informatics, LIPIcs}, volume~75,
  pages 24:1--24:9. Schloss Dagstuhl--Leibniz-Zentrum fuer Informatik, 2017.

\bibitem{Bach2013}
F.~Bach.
\newblock {Learning with submodular functions: A convex optimization
  perspective}.
\newblock {\em Found. Trends{\textregistered} Mach. Learn.}, 6(2-3):145--373,
  2013.

\bibitem{Bach2010}
F.~R. Bach.
\newblock {Structured sparsity-inducing norms through submodular functions}.
\newblock In J.~D. Lafferty, C.~K.~I. Williams, J.~Shawe-Taylor, R.~S. Zemel,
  and A.~Culotta, editors, {\em Advances in Neural Information Processing
  Systems 23}, pages 118--126. Curran Associates, Inc., 2010.

\bibitem{Bampis2012}
E.~Bampis, C.~D{\"{u}}rr, F.~Kacem, and I.~Milis.
\newblock {Speed scaling with power down scheduling for agreeable deadlines}.
\newblock {\em Sustain. Comput. Inform. Syst.}, 2(4):184--189, 2012.

\bibitem{Blum1973}
M.~Blum, R.~W. Floyd, V.~Pratt, R.~L. Rivest, and R.~E. Tarjan.
\newblock {Time bounds for selection}.
\newblock {\em J. Comput. Syst. Sci.}, 7(4):448--461, 1973.

\bibitem{Brucker1984}
P.~Brucker.
\newblock {An {$O(n)$} algorithm for quadratic knapsack problems}.
\newblock {\em Oper. Res. Lett.}, 3(3):163--166, 1984.

\bibitem{Chakrabarty2014}
D.~Chakrabarty, P.~Jain, and P.~Kothari.
\newblock {Provable submodular minimization using wolfe's algorithm}.
\newblock In Z.~Ghahramani, M.~Welling, C.~Cortes, N.~D. Lawrence, and K.~Q.
  Weinberger, editors, {\em Advances in Neural Information Processing Systems
  27}, pages 802--809. Curran Associates, Inc., 2014.

\bibitem{Combettes2018b}
P.~L. Combettes.
\newblock {Perspective functions: Properties, constructions, and examples}.
\newblock {\em Set-Valued Var. Anal.}, 26(2):247--264, 2018.

\bibitem{Combettes2018a}
P.~L. Combettes and C.~L. M{\"{u}}ller.
\newblock {Perspective functions: Proximal calculus and applications in
  high-dimensional statistics}.
\newblock {\em J. Math. Anal. Appl.}, 457(2):1283--1306, 2018.

\bibitem{Cosares1994}
S.~Cosares and D.~S. Hochbaum.
\newblock {Strongly polynomial algorithms for the quadratic transportation
  problem with a fixed number of sources}.
\newblock {\em Math. Oper. Res.}, 19(1):94--111, 1994.

\bibitem{Dai2006}
Y.-H. Dai and R.~Fletcher.
\newblock {New algorithms for singly linearly constrained quadratic programs
  subject to lower and upper bounds}.
\newblock {\em Math. Program.}, 106(3):403--421, 2006.

\bibitem{DAmico2014}
A.~A. D'Amico, L.~Sanguinetti, and D.~P. Palomar.
\newblock {Convex separable problems with linear constraints in signal
  processing and communications}.
\newblock {\em IEEE Trans. Signal Process.}, 62(22):6045--6058, 2014.

\bibitem{Frederickson1982}
G.~N. Frederickson and D.~B. Johnson.
\newblock {The complexity of selection and ranking in X + Y and matrices with
  sorted columns}.
\newblock {\em J. Comput. Syst. Sci.}, 24(2):197--208, 1982.

\bibitem{Friedrich2015}
U.~Friedrich, R.~M{\"{u}}nnich, S.~de~Vries, and M.~Wagner.
\newblock {Fast integer-valued algorithms for optimal allocations under
  constraints in stratified sampling}.
\newblock {\em Comput. Stat. Data Anal.}, 92:1--12, 2015.

\bibitem{Fujishige1980}
S.~Fujishige.
\newblock {Lexicographically optimal base of a polymatroid with respect to a
  weight vector}.
\newblock {\em Math. Oper. Res.}, 5(2):186--196, 1980.

\bibitem{Fujishige1984}
S.~Fujishige.
\newblock {Structures of polyhedra determined by submodular functions on
  crossing families}.
\newblock {\em Math. Program.}, 29(2):125--141, 1984.

\bibitem{Fujishige2005}
S.~Fujishige.
\newblock {Submodular functions and optimization}.
\newblock {\em Ann. Discret. Math.}, 58:1--395, 2005.

\bibitem{Fujishige2011}
S.~Fujishige and S.~Isotani.
\newblock {A submodular function minimization algorithm based on the
  minimum-norm base}.
\newblock {\em Pac. J. Optim.}, 7(1):3--17, 2011.

\bibitem{Gerards2014}
M.~E.~T. Gerards.
\newblock {\em {Algorithmic power management: Energy minimisation under
  real-time constraints}}.
\newblock PhD thesis, University of Twente, 2014.

\bibitem{Gerards2016aa}
M.~E.~T. Gerards, J.~L. Hurink, and P.~K.~F. H{\"{o}}lzenspies.
\newblock {A survey of offline algorithms for energy minimization under
  deadline constraints}.
\newblock {\em J. Sched.}, 19(1):3--19, 2016.

\bibitem{Gerards2014a}
M.~E.~T. Gerards, J.~L. Hurink, P.~K.~F. H{\"{o}}lzenspies, J.~Kuper, and
  G.~J.~M. Smit.
\newblock {Analytic clock frequency selection for global DVFS}.
\newblock In {\em 2014 22nd Euromicro International Conference on Parallel,
  Distributed, and Network-Based Processing}, pages 512--519, Turin, 2014.

\bibitem{Gerards2015}
M.~E.~T. Gerards, H.~A. Toersche, G.~Hoogsteen, T.~van~der Klauw, J.~L. Hurink,
  and G.~J.~M. Smit.
\newblock {Demand side management using profile steering}.
\newblock In {\em 2015 IEEE Eindhoven PowerTech}, Eindhoven, 2015. IEEE.

\bibitem{Gondzio2012}
J.~Gondzio.
\newblock {Interior point methods 25 years later}.
\newblock {\em Eur. J. Oper. Res.}, 218(3):587--601, 2012.

\bibitem{Groenevelt1991}
H.~Groenevelt.
\newblock {Two algorithms for maximizing a separable concave function over a
  polymatroid feasible region}.
\newblock {\em Eur. J. Oper. Res.}, 54(2):227--236, 1991.

\bibitem{Harks2014}
T.~Harks, M.~Klimm, and B.~Peis.
\newblock {Resource competition on integral polymatroids}.
\newblock In T.-Y. Liu, Q.~Qi, and Y.~Ye, editors, {\em 10th International
  Conference on Web and Internet Economics}, pages 189--202, Cham, 2014.
  Springer International Publishing.

\bibitem{He2013}
P.~He, L.~Zhao, S.~Zhou, and Z.~Niu.
\newblock {Water-filling: A geometric approach and its application to aolve
  generalized radio resource allocation problems}.
\newblock {\em IEEE Trans. Wirel. Commun.}, 12(7):3637--3647, 2013.

\bibitem{He2012}
S.~He, J.~Zhang, and S.~Zhang.
\newblock {Polymatroid optimization, submodularity, and joint replenishment
  games}.
\newblock {\em Oper. Res.}, 60(1):128--137, 2012.

\bibitem{Hochbaum1994}
D.~S. Hochbaum.
\newblock {Lower and upper bounds for the allocation problem and other
  nonlinear optimization problems}.
\newblock {\em Math. Oper. Res.}, 19(2):390--409, 1994.

\bibitem{Hochbaum1995}
D.~S. Hochbaum and S.-P. Hong.
\newblock {About strongly polynomial time algorithms for quadratic optimization
  over submodular constraints}.
\newblock {\em Math. Program.}, 69:269--309, 1995.

\bibitem{Hochbaum1990}
D.~S. Hochbaum and J.~G. Shanthikumar.
\newblock {Convex separable optimization is not much harder than linear
  optimization}.
\newblock {\em J. ACM}, 37(4):843--862, 1990.

\bibitem{Huang2009}
W.~Huang and Y.~Wang.
\newblock {An optimal speed control scheme supported by media servers for
  low-power multimedia applications}.
\newblock {\em Multimed. Syst.}, 15(2):113--124, 2009.

\bibitem{Hvattum2013}
L.~M. Hvattum, I.~Norstad, K.~Fagerholt, and G.~Laporte.
\newblock {Analysis of an exact algorithm for the vessel speed optimization
  problem}.
\newblock {\em Netw.}, 62(2):132--135, 2013.

\bibitem{Ibaraki1988}
T.~Ibaraki and N.~Katoh.
\newblock {\em {Resource allocation problems: Algorithmic approaches}}.
\newblock The MIT Press, Cambridge, MA, 1 edition, 1988.

\bibitem{Irani2007}
S.~Irani, S.~Shukla, and R.~Gupta.
\newblock {Algorithms for power savings}.
\newblock {\em ACM Trans. Algorithms}, 3(4):41:1--41:23, 2007.

\bibitem{Jain2010}
K.~Jain and V.~V. Vazirani.
\newblock {Eisenberg–Gale markets: Algorithms and game-theoretic properties}.
\newblock {\em Games Econ. Behav.}, 70(1):84--106, 2010.

\bibitem{Kaplan2019}
H.~Kaplan, L.~Kozma, O.~Zamir, and U.~Zwick.
\newblock {Selection from heaps, row-sorted matrices, and X+Y using soft
  heaps}.
\newblock In J.~T. Fineman and M.~Mitzenmacher, editors, {\em 2nd Symposium on
  Simplicity in Algorithms (SOSA 2019)}, pages 5:1--5:21, San Diego, 2019.
  Schloss Dagstuhl--Leibniz-Zentrum fuer Informatik.

\bibitem{Katoh2013}
N.~Katoh, A.~Shioura, and T.~Ibaraki.
\newblock {Resource allocation problems}.
\newblock In P.~M. Pardalos, D.-Z. Du, and R.~L. Graham, editors, {\em Handbook
  of Combinatorial Optimization}, pages 2897--2988. Springer, New York, NY, 2
  edition, 2013.

\bibitem{Khakurel2014}
S.~Khakurel, C.~Leung, and T.~Le-Ngoc.
\newblock {A generalized water-filling algorithm with linear complexity and
  finite convergence time}.
\newblock {\em IEEE Wirel. Commun. Lett.}, 3(2):225--228, 2014.

\bibitem{Kiwiel2005}
K.~C. Kiwiel.
\newblock {On Floyd and Rivest's SELECT algorithm}.
\newblock {\em Theor. Comput. Sci.}, 347(1):214--238, 2005.

\bibitem{Kiwiel2008a}
K.~C. Kiwiel.
\newblock {Breakpoint searching algorithms for the continuous quadratic
  knapsack problem}.
\newblock {\em Math. Program.}, 112(2):473--491, 2007.

\bibitem{Kiwiel2008b}
K.~C. Kiwiel.
\newblock {Variable fixing algorithms for the continuous quadratic knapsack
  problem}.
\newblock {\em J. Optim. Theory Appl.}, 136(3):445--458, mar 2008.

\bibitem{Ling2012}
X.~Ling, B.~Wu, P.~Ho, F.~Luo, and L.~Pan.
\newblock {Fast water-filling for agile power allocation in multi-channel
  wireless communications}.
\newblock {\em IEEE Commun. Lett.}, 16(8):1212--1215, 2012.

\bibitem{Liu2020}
S.~Liu.
\newblock {A review for submodular optimization on machine scheduling
  problems}.
\newblock In D.-Z. Du and J.~Wang, editors, {\em Complexity and Approximation:
  In Memory of Ker-I Ko}, pages 252--267. Springer International Publishing,
  Cham, 2020.

\bibitem{Lobo2007}
M.~S. Lobo, M.~Fazel, and S.~Boyd.
\newblock {Portfolio optimization with linear and fixed transaction costs}.
\newblock {\em Ann. Oper. Res.}, 152(1):341--365, 2007.

\bibitem{Lund2016}
H.~Lund, P.~A. {\O}stergaard, D.~Connolly, I.~Ridjan, B.~V. Mathiesen,
  F.~Hvelplund, J.~Z. Thellufsen, and P.~Sorkn{\ae}s.
\newblock {Energy storage and smart energy systems}.
\newblock {\em Int. J. Sustain. Energy Plan. Manag.}, 11:3--14, 2016.

\bibitem{Mairal2011}
J.~Mairal, R.~Jenatton, G.~Obozinski, and F.~Bach.
\newblock {Convex and network flow optimization for structured sparsity}.
\newblock {\em J. Mach. Learn. Res.}, 12(81):2681--2720, 2011.

\bibitem{Megiddo1993}
N.~Megiddo and A.~Tamir.
\newblock {Linear time algorithms for some separable quadratic programming
  problems}.
\newblock {\em Oper. Res. Lett.}, 13(4):203--211, 1993.

\bibitem{Meng2013}
X.~Meng.
\newblock {Scalable simple random sampling and stratified sampling}.
\newblock In S.~Dasgupta and D.~McAllester, editors, {\em Proceedings of the
  30th International Conference on Machine Learning}, volume~28 of {\em
  Proceedings of Machine Learning Research}, pages 531--539, Atlanta, Georgia,
  2013. PMLR.

\bibitem{Moriguchi2004}
S.~Moriguchi and A.~Shioura.
\newblock {On Hochbaum's proximity-scaling algorithm for the general resource
  allocation problem}.
\newblock {\em Math. Oper. Res.}, 29(2):394--397, 2004.

\bibitem{Moriguchi2011}
S.~Moriguchi, A.~Shioura, and N.~Tsuchimura.
\newblock {M-convex function minimization by continuous relaxation approach:
  proximity theorem and algorithm}.
\newblock {\em SIAM J. Optim.}, 21(3):633--668, 2011.

\bibitem{Muller2011}
M.~O. M{\"{u}}ller, A.~St{\"{a}}mpfli, U.~Dold, and T.~Hammer.
\newblock {Energy autarky: A conceptual framework for sustainable regional
  development}.
\newblock {\em Energy Policy}, 39(10):5800--5810, 2011.

\bibitem{Muller-Hannemann2010}
M.~M{\"u}ller-Hannemann and S.~Schirra, editors.
\newblock {\em {Algorithm engineering: Bridging the gap between algorithm
  theory and practice}}.
\newblock Springer Berlin Heidelberg, Berlin, Heidelberg, 1 edition, 2010.

\bibitem{Munkhammar2013}
J.~Munkhammar, P.~Grahn, and J.~Wid{\'{e}}n.
\newblock {Quantifying self-consumption of on-site photovoltaic power
  generation in households with electric vehicle home charging}.
\newblock {\em Sol. Energy}, 97:208--216, 2013.

\bibitem{Nagano2012}
K.~Nagano and K.~Aihara.
\newblock {Equivalence of convex minimization problems over base polytopes}.
\newblock {\em Jpn. J. Ind. Appl. Math.}, 29(3):519--534, 2012.

\bibitem{Nagano2013}
K.~Nagano and Y.~Kawahara.
\newblock {Structured convex optimization under submodular constraints}.
\newblock In {\em Proceedings of the Twenty-Ninth Conference Annual Conference
  on Uncertainty in Artificial Intelligence (UAI-13)}, pages 459--468,
  Corvallis, Oregon, 2013. AUAI Press.

\bibitem{Neyman1934}
J.~Neyman.
\newblock {On the two different aspects of the representative method: The
  method of stratified sampling and the method of purposive selection}.
\newblock {\em J. R. Stat. Soc.}, 97(4):558--606, 1934.

\bibitem{Norstad2011}
I.~Norstad, K.~Fagerholt, and G.~Laporte.
\newblock {Tramp ship routing and scheduling with speed optimization}.
\newblock {\em Transp. Res. Part C Emerg. Technol.}, 19(5):853--865, 2011.

\bibitem{Orlin2013}
J.~B. Orlin and B.~Vaidyanathan.
\newblock {Fast algorithms for convex cost flow problems on circles, lines, and
  trees}.
\newblock {\em Netw.}, 62(4):288--296, 2013.

\bibitem{Palomar2005}
D.~P. Palomar and J.~R. Fonollosa.
\newblock {Practical algorithms for a family of waterfilling solutions}.
\newblock {\em IEEE Trans. Signal Process.}, 53(2):686--695, 2005.

\bibitem{Patriksson2008}
M.~Patriksson.
\newblock {A survey on the continuous nonlinear resource allocation problem}.
\newblock {\em Eur. J. Oper. Res.}, 185(1):1--46, 2008.

\bibitem{Patriksson2015}
M.~Patriksson and C.~Str{\"{o}}mberg.
\newblock {Algorithms for the continuous nonlinear resource allocation problem
  - new implementations and numerical studies}.
\newblock {\em Eur. J. Oper. Res.}, 243(3):703--722, 2015.

\bibitem{Psaraftis2014}
H.~N. Psaraftis and C.~A. Kontovas.
\newblock {Ship speed optimization: Concepts, models and combined speed-routing
  scenarios}.
\newblock {\em Transp. Res. Part C Emerg. Technol.}, 44:52--69, 2014.

\bibitem{Reijnders2020}
V.~M. J.~J. Reijnders, M.~E.~T. Gerards, and J.~L. Hurink.
\newblock {A hybrid pricing mechanism for joint system optimization and social
  acceptance}, 2020.
\newblock Accepted for ENERGYCON 2020, Tunis.

\bibitem{Roberts2011}
B.~P. Roberts and C.~Sandberg.
\newblock {The role of energy storage in development of smart grids}.
\newblock {\em Proc. IEEE}, 99(6):1139--1144, 2011.

\bibitem{Romeijn2007}
H.~E. Romeijn, J.~Geunes, and K.~Taaffe.
\newblock {On a nonseparable convex maximization problem with continuous
  knapsack constraints}.
\newblock {\em Oper. Res. Lett.}, 35(2):172--180, 2007.

\bibitem{Sanathanan1971}
L.~Sanathanan.
\newblock {On an allocation problem with multistage constraints}.
\newblock {\em Oper. Res.}, 19(7):1647--1663, 1971.

\bibitem{SchootUiterkamp2020b}
M.~H.~H. {Schoot Uiterkamp}, M.~E.~T. Gerards, and J.~L. Hurink.
\newblock {A fast algorithm for the quadratic resource allocation problem with
  nested constraints}.
\newblock Working paper, 2020.

\bibitem{SchootUiterkamp2019a}
M.~H.~H. {Schoot Uiterkamp}, M.~E.~T. Gerards, and J.~L. Hurink.
\newblock {Quadratic nonseparable resource allocation problems with generalized
  bound constraints}, 2020.
\newblock arXiv: \url{https://arxiv.org/abs/2007.06280}.

\bibitem{SchootUiterkamp2018b}
M.~H.~H. {Schoot Uiterkamp}, T.~van~der Klauw, M.~E.~T. Gerards, and J.~L.
  Hurink.
\newblock {Offline and online scheduling of electric vehicle charging with a
  minimum charging threshold}.
\newblock In {\em 2018 IEEE International Conference on Communications,
  Control, and Computing Technologies for Smart Grids}, Aalborg, 2018.

\bibitem{Shams2014}
F.~Shams, G.~Bacci, and M.~Luise.
\newblock {A survey on resource allocation techniques in OFDM(A) networks}.
\newblock {\em Comput. Netw.}, 65:129--150, 2014.

\bibitem{Sharkey2011}
T.~C. Sharkey, H.~E. Romeijn, and J.~Geunes.
\newblock {A class of nonlinear nonseparable continuous knapsack and
  multiple-choice knapsack problems}.
\newblock {\em Math. Program.}, 126(1):69--96, 2011.

\bibitem{Shioura2017}
A.~Shioura, N.~V. Shakhlevich, and V.~A. Strusevich.
\newblock {Machine speed scaling by adapting methods for convex optimization
  with submodular constraints}.
\newblock {\em INFORMS J. Comput.}, 29(4):724--736, 2017.

\bibitem{Shioura2018}
A.~Shioura, N.~V. Shakhlevich, and V.~A. Strusevich.
\newblock {Preemptive models of scheduling with controllable processing times
  and of scheduling with imprecise computation: A review of solution
  approaches}.
\newblock {\em Eur. J. Oper. Res.}, 266(3):795--818, 2018.

\bibitem{Slager2019}
J.~Slager.
\newblock {\em {Nonlinear convex optimisation problems in the smart grid}}.
\newblock B.sc. thesis, University of Twente, 2019.

\bibitem{Sun2013}
X.~Sun, X.~Zheng, and D.~Li.
\newblock {Recent advances in mathematical programming with semi-continuous
  variables and cardinality constraint}.
\newblock {\em J. Oper. Res. Soc. China}, 1(1):55--77, 2013.

\bibitem{Tamir1993}
A.~Tamir.
\newblock {A strongly polynomial algorithm for minimum convex separable
  quadratic cost flow problems on two-terminal series—parallel networks}.
\newblock {\em Math. Program.}, 59(1):117--132, 1993.

\bibitem{Uddin2018}
M.~Uddin, M.~F. Romlie, M.~F. Abdullah, S.~{Abd Halim}, A.~H. {Abu Bakar}, and
  T.~{Chia Kwang}.
\newblock {A review on peak load shaving strategies}.
\newblock {\em Renew. Sustain. Energy Rev.}, 82:3323--3332, 2018.

\bibitem{vdKlauw2017}
T.~van~der Klauw, M.~E.~T. Gerards, and J.~L. Hurink.
\newblock {Resource allocation problems in decentralized energy management}.
\newblock {\em OR Spectr.}, 39(3):749--773, 2017.

\bibitem{Vanhoudt2014}
D.~Vanhoudt, D.~Geysen, B.~Claessens, F.~Leemans, L.~Jespers, and J.~{Van
  Bael}.
\newblock {An actively controlled residential heat pump: Potential on peak
  shaving and maximization of self-consumption of renewable energy}.
\newblock {\em Renew. Energy}, 63:531--543, 2014.

\bibitem{Veinott1971}
A.~F. Veinott.
\newblock {Least d-majorized network flows with inventory and statistical
  applications}.
\newblock {\em Manag. Sci.}, 17(9):547--567, 1971.

\bibitem{Vidal2018}
T.~Vidal, D.~Gribel, and P.~Jaillet.
\newblock {Separable convex optimization with nested lower and upper
  constraints}.
\newblock {\em INFORMS J. Optim.}, 1(1):71--90, 2019.

\bibitem{Vidal2016}
T.~Vidal, P.~Jaillet, and N.~Maculan.
\newblock {A decomposition algorithm for nested resource allocation problems}.
\newblock {\em SIAM J. Optim.}, 26(2):1322--1340, 2016.

\bibitem{Wolfe1976}
P.~Wolfe.
\newblock {Finding the nearest point in a polytope}.
\newblock {\em Math. Program.}, 11(1):128--149, 1976.

\bibitem{Wright2020}
S.~E. Wright and S.~Lim.
\newblock {Solving nested-constraint resource allocation problems with an
  interior point method}.
\newblock {\em Oper. Res. Lett.}, 48(3):297--303, 2020.

\bibitem{Wright2014}
S.~E. Wright and J.~J. Rohal.
\newblock {Solving the continuous nonlinear resource allocation problem with an
  interior point method}.
\newblock {\em Oper. Res. Lett.}, 42(6):404--408, 2014.

\bibitem{Wu2019}
Z.~Wu.
\newblock {\em {Fast exact algorithms for optimization problems in resource
  allocation and switched linear systems}}.
\newblock PhD thesis, University of Minesota, 2019.

\bibitem{Xing2020}
C.~Xing, Y.~Jing, S.~Wang, S.~Ma, and H.~V. Poor.
\newblock {New viewpoint and algorithms for water-filling solutions in wireless
  communications}.
\newblock {\em IEEE Trans. Signal Process.}, 68:1618--1634, 2020.

\bibitem{Zame2018}
K.~K. Zame, C.~A. Brehm, A.~T. Nitica, C.~L. Richard, and G.~D. {Schweitzer
  III}.
\newblock {Smart grid and energy storage: Policy recommendations}.
\newblock {\em Renew. Sustain. Energy Rev.}, 82:1646--1654, 2018.

\bibitem{Zhuravlev2013}
S.~Zhuravlev, J.~C. Saez, S.~Blagodurov, A.~Fedorova, and M.~Prieto.
\newblock {Survey of energy-cognizant scheduling techniques}.
\newblock {\em IEEE Trans. Parallel Distrib. Syst.}, 24(7):1447--1464, 2013.

\end{thebibliography}

\end{document}